\documentclass[10pt,reqno]{article}

\usepackage{graphicx}
\usepackage{amscd}
\usepackage{graphics,amsmath,amssymb,amsthm}
\usepackage{amsthm}
\usepackage{amsfonts,enumerate}
\usepackage{latexsym}
\usepackage{mathrsfs}

 \usepackage{lineno,hyperref}

 \usepackage[mathscr]{eucal}
 
 \usepackage{pgfplots}
 \usepackage{tikz}
 \usetikzlibrary{arrows,shapes}
 \usepackage{xifthen}
 \usepackage{fancyhdr}
 \usepackage{lmodern}
 \usepackage[utf8]{inputenc}
 \usepackage[T1]{fontenc}
 \usepackage{cite}

\newtheorem{theorem}{Theorem}
\theoremstyle{plain}

\newtheorem{corollary}{Corollary}

\newtheorem{definition}{Definition}

\newtheorem{lemma}{Lemma}

\numberwithin{equation}{section}

\begin{document}

\begin{center}
{\LARGE \bf
Incomplete $q$-Chebyshev Polynomials}\\
\vskip 1cm

{\large Elif ERCAN$^{1}$, Mirac CETIN FIRENGIZ$^{2}$ and Naim TUGLU$^{1}$} \\
\vskip 0.5cm

$^{1}$Gazi University, Faculty of Art and Science, Department of Mathematics,\smallskip \\
Teknikokullar 06500, Ankara-Turkey \smallskip \\
{\tt  elifercan06@gmail.com } \ \ \ \ \ {\tt naimtuglu@gazi.edu.tr}\\
\vskip 0.5cm

$^{2}$Ba\c{s}kent University, Department of Mathematics Education, \smallskip\\  Ankara-Turkey \smallskip \\
{\tt mcetin@baskent.edu.tr} 
\vskip 0.8cm
\end{center}


\begin{abstract}
	In this paper, we get the generating functions of $q$-Chebyshev polynomials using  $\eta _{z}$ operator, which is $\eta _{z}\left( f(z)\right)=f(qz)$ for any given function $f\left( z\right) $. Also considering explicit formulas of $q$-Chebyshev polynomials, we give new generalizations of $q$-Chebyshev polynomials called incomplete $q$-Chebyshev polynomials of the first and second kind. We obtain recurrence relations and several properties of these polynomials. We show that there are connections between incomplete $q$-Chebyshev polynomials and the some well-known polynomials.\smallskip   \smallskip \\
\textbf{Mathematics Subject Classiffication:} 11B39, 05A30
\end{abstract}

\thispagestyle{empty}

\section{Introduction}

Chebyshev polynomials are of great importance in many area of mathematics, particularly approximation theory. The Chebyshev polynomials of the second kind can be expressed by the formula
\begin{equation*}
U_{n}(x)=2xU_{n-1}(x)-U_{n-2}(x)\ \ \ \ n\geq 2,
\end{equation*}
with initial conditions $U_{0}=1,\ U_{1}(x)=2x$ and the Chebyshev polynomials of the first kind can be defined as
\begin{equation*}
T_{n}(x)=2xT_{n-1}(x)-T_{n-2}(x)\ \ \ \ n\geq 2,
\end{equation*}
with initial conditions $T_{0}(x)=1,\;T_{1}(x)=x$ in \cite{12}.

The well-known Fibonacci and Lucas sequences are defined by the recurrence relations
\begin{equation*}
F_{n+1}=F_{n}+F_{n-1}\ \ \ \ \ n\geq 1
\end{equation*}%
\begin{equation*}
L_{n+1}=L_{n}+L_{n-1}\ \ \ \ \ n\geq 1
\end{equation*}%
with initial conditions $F_{0}=0,\ F_{1}=1$ and $L_{0}=2,\ L_{1}=1$, respectively. In \cite{11}, Filipponi introduced a generalization of the Fibonacci numbers. Accordingly, the incomplete Fibonacci and Lucas numbers are determined by:
\begin{equation}
F_{n}(k)=\sum_{j=0}^{k}\binom{n-1-j}{j},\ \ \ \ \ \ \ \ \ \ \ \ \ 0\leq k\leq \left\lfloor \tfrac{n-1}{2}\right\rfloor  \label{1.1}
\end{equation}
and
\begin{equation}
L_{n}(k)=\sum_{j=0}^{k}\frac{n}{n-j}\binom{n-j}{j},  \ \ \ \ \ \ \ \ \ \ \ \ \ \ 0\leq k\leq \left\lfloor \tfrac{n}{2}\right\rfloor  \label{1.2}
\end{equation}
where $n\in \mathbb{N}$. Note that $F_{n}(\left\lfloor \tfrac{n-1}{2}\right\rfloor )=F_{n}$ and $L_{n}(\left\lfloor \tfrac{n}{2}\right\rfloor )=L_{n}$. 
In \cite{13}, the generating functions of incomplete Fibonacci and Lucas polynomials were given by Pint\'{e}r and Srivastava. For more results on the incomplete Fibonacci numbers, the readers may refer to \cite{9,10,14,15,16}.

We need $q$-integer and $q$-binomial coefficient. There are several equivalent definition and notation for the $q$-binomial coefficients \cite{1,18,19,20}. Let $q\in \mathbb{C}$ with $0<\left\vert q\right\vert <1$ \ as an indeterminate and nonnegative integer $n$. The $q$-integer of the number $n$ is defined by
\begin{equation*}
\left[ n\right] _{q}:=\dfrac{1-q^{n}}{1-q}, 
\end{equation*}
with $\left[ 0\right] _{q}=0$. The $q$-factorial is defined by
\begin{equation*}
\left[ n\right]!:=\left\{
\begin{array}{cl}
\left[n\right] _{q}\,\left[ n-1\right] _{q}\dots \left[ 1	\right] _{q}, & \text{\ if \ }n=1,2,\dots \medskip \\
1, & \text{\ if\ \ }n=0.
\end{array}
\right.
\end{equation*}
The Gaussian or $q-$binomial coefficients are defined by
\begin{equation*}
\biggl[\displaystyle{n\atop k}\biggr] _{q}:=\frac{\left[ n\right] _{q} !}{\left[ n-k\right] _{q}!\,\left[k\right] _{q}!}, \ \ \ \ \  0\leq k\leq n
\end{equation*}
or
\begin{equation*}
\biggl[\displaystyle{n\atop k}\biggr] _{q}:=\frac{(q;q)_{n}}{(q;q)_{n-k}(q;q)_{k}},\ \ \ \ \
0\leq k\leq n
\end{equation*}   		
with ${n \brack k}=0$ \ for $n<k$, where $(x;q)_{n}$ is the $q-$shifted factorial, that is, $(x;q)_{0}=1$,
\begin{equation*}
\left( x;q\right) _{n}=\prod\limits_{i=0}^{n-1}\left( 1-q^{i}x\right) .
\end{equation*}
The $q$-binomial coefficient satisfies the recurrence relations and properties:
\begin{eqnarray}
\biggl[\displaystyle{n+1\atop k}\biggr] _{q}&=& q^{k} \biggl[\displaystyle{n\atop k}\biggr] _{q}+\biggl[\displaystyle{n\atop k-1}\biggr] _{q}  \label{1.2a} \\
\biggl[\displaystyle{n+1\atop k}\biggr] _{q}&=& \biggl[\displaystyle{n\atop k}\biggr] _{q}+ q^{n-k+1} \biggl[\displaystyle{n\atop k-1}\biggr] _{q}  \label{1.2b} \\
\frac{\left[ n\right] _{q} }{\left[n-k\right] _{q} } \biggl[\displaystyle{n-k\atop k}\biggr] _{q}&=& q^{k}\biggl[\displaystyle{n-k\atop k}\biggr] _{q}+\biggl[\displaystyle{n-k-1\atop k-1}\biggr] _{q} \label{1.2c} \\
q^{k}\frac{\left[ n\right] _{q} }{\left[n-k\right] _{q} } \biggl[\displaystyle{n-k\atop k}\biggr] _{q}&=& q^{k}\biggl[\displaystyle{n-k\atop k}\biggr] _{q}+q^{n}\biggl[\displaystyle{n-k-1\atop k-1}\biggr] _{q}. \label{1.2d}
\end{eqnarray}

The q-analogues of Fibonacci polynomials are studied by Carlitz in \cite{2}. Also, a new $q$-analogue of the Fibonacci polynomials is defined by Cigler and obtain some of its properties in \cite{5}. In \cite{6}, Pan study some arithmetic properties of the q-Fibonacci numbers and the q-Pell numbers. Cigler defined $q$-analogues of the Chebyshev polynomials and study properties of these polynomials in \cite{3,4}.

In this paper, we derive generating functions of $q$-Chebyshev polynomials of first and second kind. More generally, we define incomplete $q$-Chebyshev polynomials of first and second kind. We get recurrence relations and several properties of these polynomials. We show that there are the relationships between $q$-Chebyshev polynomials and incomplete $q$-Chebyshev polynomials.

\section{$q$-Chebyshev Polynomials}

\begin{definition}
	The $q$-Chebyshev polynomials of the second kind are defined by
	\begin{equation}
	\mathcal{U}_{n}(x,s,q)=(1+q^{n})x\: \mathcal{U}_{n-1}(x,s,q)+q^{n-1}s\: \mathcal{U}_{n-2}(x,s,q) \ \  \ \ \ \ n\geq 2, \label{1.3}
	\end{equation}%
	with initial conditions $\mathcal{U}_{0}(x,s,q)=1$ and $\mathcal{U}_{1}(x,s,q)=(1+q)x$ in \cite{3}. 
\end{definition}

\begin{definition}
	The $q$-Chebyshev polynomials of the first kind are defined by
	\begin{equation}
	\mathcal{T}_{n}(x,s,q)=(1+q^{n-1})x\mathcal{T}_{n-1}(x,s,q)+q^{n-1}s\mathcal{T}_{n-2}(x,s,q) \ \ \ \ \ \ n\geq 2,  \label{1.4}
	\end{equation}%
	with initial conditions $\mathcal{T}_{0}(x,s,q)=1$ and $\mathcal{T}_{1}(x,s,q)=x$ in \cite{3}. 
\end{definition}
It is clear that $\:\mathcal{U}_{n}(x,-1,1)=U_{n}(x)$ and $\mathcal{T}_{n}(x,-1,1)=T_{n}(x)$. The $q$-Chebyshev polynomials of the second kind is determined as the
combinatorial sum
\begin{equation}
\:\mathcal{U}_{n}(x,s,q)=\sum\limits_{j=0}^{\lfloor \frac{n}{2}\rfloor }q^{j^{2}}\biggl[\displaystyle{n-j \atop j }\biggr]_{q}\frac{(-q;q)_{n-j}}{(-q;q)_{j}}\:s^{j}x^{n-2j},\ \ \ \ \ \ n\geq 0  \label{1.5}
\end{equation}
and the $q$-Chebyshev polynomials of the first kind is determined as
\begin{equation}
\mathcal{T}_{n}(x,s,q)=\sum\limits_{j=0}^{\lfloor \frac{n}{2}\rfloor }q^{j^{2}}\frac{\left[ n\right] _{q}}{\left[ n-j\right] _{q}}\biggl[\displaystyle{n-j \atop j }\biggr]_{q}\frac{(-q;q)_{n-j-1}}{(-q;q)_{j}}\:s^{j}x^{n-2j}, \ \ \ \ \ n>0  \label{1.6}
\end{equation}%
with $\mathcal{T}_{0}(x,s,q)=1$ in \cite{3}.

	\begin{table}[h!]
		\caption{Some special cases of the $q$-Chebyshev polynomials of the second kind}
		
		\begin{equation*}
			\begin{tabular}{|l|c|l|l|l|}
				\hline
				$x$ & $s$ & $q$ & $\:\mathcal{U}_{n}(x,s,q)$ & $q-$Chebyshev polynomials of the second kind
				\\ \hline\hline
				$x$ & $-1$ & $1$ & $U_{n}(x)$ & \. Chebyshev polynomials of the second kind 
				\\ \hline
				$\frac{x}{2}$ & $1$ & $1$ & $F_{n+1}(x)$ & Fibonacci polynomials  \\ \hline
				$\frac{1}{2}$ & $1$ & $1$ & $F_{n+1}$ & Fibonacci numbers  \\ \hline
				$x$ & $1$ & $1$ & $P_{n+1}(x)$ & Pell polynomials  \\ \hline
				$1$ & $1$ & $1$ & $P_{n+1}$ & Pell numbers  \\ \hline
				$\frac{1}{2}$ & $2y$ & $1$ & $J_{n+1}(y)$ & Jacobsthal polynomials  \\ \hline
				$\frac{1}{2}$ & $2$ & $1$ & $J_{n+1}$ & Jacobsthal numbers  \\ \hline
			\end{tabular}%
		\end{equation*}
	\end{table}
	
	
	\begin{table}[h!]
		\caption{Some special cases of the $q$-Chebyshev polynomials of the first kind}
		
		\begin{equation*}
			\begin{tabular}{|l|c|l|l|l|}
				\hline
				$x$ & $s$ & $q$ & $\mathcal{T}_{n}(x,s,q)$ & $q-$Chebshev polynomials of the first kind
				\\ \hline\hline
				$x$ & $-1$ & $1$ &  $T_{n}(x)$ & Chebyshev polynomials of the first kind
				\\ \hline
				$\frac{x}{2}$ & $1$ & $1$ & $\frac{1}{2}L_{n}(x)$ & Lucas polynomials \\ \hline
				$\frac{1}{2}$ & $1$ & $1$ & $\frac{1}{2}~L_{n}$ & Lucas numbers \\ \hline
				$x$ & $1$ & $1$ & $\frac{1}{2}~Q_{n}(x)$ & Pell-Lucas polynomials \\ \hline
				$1$ & $1$ & $1$ &  $\frac{1}{2}~Q_{n}$ & Pell-Lucas numbers \\ \hline
				$\frac{1}{2}$ & $2y$ & $1$ &  $\frac{1}{2}~j_{n}(y)$ & Jacobsthal-Lucas polynomials \\ \hline
				$\frac{1}{2}$ & $2$ & $1$ & $\frac{1}{2}~j_{n}$ & Jacobsthal-Lucas numbers \\ \hline
			\end{tabular}%
		\end{equation*}
	\end{table}
	


\subsection{Generating Functions of $q$-Chebyshev Polynomials}

Andrews \cite{17} obtain the generating function for Schur's polynomials, which is defined by $S_{n}(q)=S_{n-1}(q)-q^{n-2}S_{n-2}(q)$ for $n>1$ with intial conditions $S_{0}(q)=0$ and $S_{1}(q)=1$. The generating funtions of $S_{n}(q)$ is
\begin{equation}
\sum_{n=0}^{\infty }S_{n}(q)x^{n}=\frac{x}{1-x-x^{2}\,\eta _{z}} 
\end{equation}
where is $\eta _{z}$ is an operator on functions of $z$ defined by $\eta _{z}\left( f(z)\right)=f(qz)$ in \cite{17}.
We give the following theorems for generating functions of $q$-Chebyshev polynomials of the second and first kind with an operator $\eta _{z}$.

\begin{theorem}
	The generating function of $q$-Chebyshev polynomials of the second kind is
	\begin{equation}
	G(z)=\frac{1}{1-zx-(xqz+sqz^{2})\,\eta _{z}}.  \label{2.1}
	\end{equation}
\end{theorem}

\begin{proof}
	Let
	\begin{equation*}
	G(z)=\sum_{n=0}^{\infty }\:\mathcal{U}_{n}z^{n}.
	\end{equation*}
	Now we show that
	\begin{equation*}
	\left( 1-xz-\left( xqz+sqz^{2}\right) \eta _{z}\right) G(z)=1.
	\end{equation*}
	Thus we write
	\begin{align*}
		\left( 1-xz-\left( xqz+sqz^{2}\right) \eta _{z}\right) G(z) =& \ \sum_{n=0}^{\infty }\:\mathcal{U}_{n}z^{n}-x\sum_{n=0}^{\infty}\:\mathcal{U}_{n}z^{n+1}-x\sum_{n=0}^{\infty}\:\mathcal{U}_{n}q^{n+1}z^{n+1}-s\sum_{n=0}^{\infty }\:\mathcal{U}_{n}q^{n+1}z^{n+2}\\
		=& \ \sum_{n=0}^{\infty }\:\mathcal{U}_{n}z^{n}-x\sum_{n=1}^{\infty }\left(1+q^{n}\right) \:\mathcal{U}_{n-1}z^{n}-s\sum_{n=2}^{\infty}\:\mathcal{U}_{n-2}q^{n-1}z^{n} \\
		=& \ \mathcal{U}_{0}+\:\mathcal{U}_{1}z-x\left( 1+q\right) \:\mathcal{U}_{0}z +\sum_{n=2}^{\infty }\left( \:\mathcal{U}_{n}-x\left( 1+q^{n}\right)\:\mathcal{U}_{n-1}-s\:\mathcal{U}_{n-2}q^{n-1}\right) z^{n}.
	\end{align*}
	Therefore we have from Eq. (\ref{1.3})
	\begin{equation*}
	\left( 1-xz-\left( xqz+sqz^{2}\right) \eta _{z}\right) G(z) =\:\mathcal{U}_{0}+\:\mathcal{U}_{1}z-x\left( 1+q\right) \:\mathcal{U}_{0}z.
	\end{equation*}
	From $\:\mathcal{U}_{0}=1$ ve $\:\mathcal{U}_{1}=(1+q)x$, we get
	\begin{equation*}
	\left( 1-xz-\left( xqz+sqz^{2}\right) \eta _{z}\right) G(z)=1.
	\end{equation*}
\end{proof}

\begin{theorem}
	The generating function of $q$-Chebyshev polynomials of the first kind is
	\begin{equation}
	S(z)=\frac{1-xz}{1-xz-(xz-sqz^{2})\,\eta _{z}}.  \label{2.2}
	\end{equation}
\end{theorem}

\begin{proof}
	Let $S(z)=\sum_{n=0}^{\infty }\mathcal{T}_{n}z^{n}.$ Then
	\begin{align*}
		\left( 1-xz-\left( xz-sqz^{2}\right) \eta _{z}\right) S(z) =& \ \sum_{n=0}^{\infty }\mathcal{T}_{n}z^{n}-x\sum_{n=1}^{\infty}\mathcal{T}_{n-1}z^{n}-x\sum_{n=1}^{\infty}\mathcal{T}_{n-1}\,q^{n-1}z^{n}-s\sum_{n=2}^{\infty }\mathcal{T}_{n-2}\,q^{n-1}z^{n}\\
		=& \ \sum_{n=0}^{\infty }\mathcal{T}_{n}\,z^{n}-x\sum_{n=1}^{\infty }\left(1+q^{n-1}\right) \mathcal{T}_{n-1}\,z^{n}-s\sum_{n=2}^{\infty}\mathcal{T}_{n-2}\,q^{n-1}z^{n} \\
		=& \ \mathcal{T}_{0}+\mathcal{T}_{1}z-2x \,\mathcal{T}_{0}\,z +\sum_{n=2}^{\infty }\left( \mathcal{T}_{n}-x\left( 1+q^{n-1}\right)\mathcal{T}_{n-1}-sq^{n-1}\,\mathcal{T}_{n-2}\right) z^{n}
	\end{align*}
	and we get using Eq. (\ref{1.4})
	\begin{equation*}
	\left( 1-xz-\left( xz-sqz^{2}\right) \eta _{z}\right) S(z)=\mathcal{T}_{0}+\mathcal{T}_{1}\,z-2\mathcal{T}_{0}\,xz.
	\end{equation*}
	From $\mathcal{T}_{0}=1$ ve $\mathcal{T}_{1}=x$, we conclude that
	\begin{equation*}
	S(z)-xzS(z)-xz\,\eta _{z}S(z)-sqz^{2}\,\eta _{z}S(z)=1-xz,
	\end{equation*}
	finally we obtain
	\begin{equation}
	S(z)=\frac{1-xz}{1-xz-(xz-sqz^{2})\,\eta _{z}}.
	\end{equation}
\end{proof}
\section{Incomplete $q$-Chebyshev Polynomials}
In this section, we define incomplete $q$-Chebyshev polynomials of the first and second kind. We give several properties for these polynomials.
\begin{definition}
	For $n$ is a nonnegative integer, the incomplete $q$-Chebyshev polynomials of the second kind are defined as
	\begin{equation}
	\:\mathcal{U}_{n}^{k}(x,s,q)=\sum\limits_{j=0}^{k}q^{j^{2}}\biggl[\displaystyle{n-j \atop j }\biggr]_{q}\frac{(-q;q)_{n-j}}{(-q;q)_{j}}\:s^{j}x^{n-2j}\ \ \ \ \ 0\leq k\leq\left\lfloor \tfrac{n}{2}\right\rfloor .  \label{3.1}
	\end{equation}
\end{definition}

When $k=\left\lfloor \frac{n}{2}\right\rfloor $ in (\ref{3.1}),\ $\:\mathcal{U}_{n}^{k}(x,s,q)=\:\mathcal{U}_{n}(x,s,q)$, we get the $q$-Chebyshev polynomials of the second kind in \cite{3,4}. Some special cases of the incomplete $q$-Chebyshev polynomials of the second kind are provided in Table 1. 

\begin{definition}
	For $n$ is a nonnegative integer, the incomplete $q$-Chebyshev polynomials of the first kind are defined by
	\begin{equation}
	\mathcal{T}_{n}^{k}(x,s,q)=\sum\limits_{j=0}^{k}q^{j^{2}}\frac{\left[ n\right] _{q}}{\left[ n-j\right] _{q}}\biggl[\displaystyle{n-j \atop j }\biggr]_{q}\frac{(-q;q)_{n-j-1}}{(-q;q)_{j}}\:s^{j}x^{n-2j} \ \ \ \ 0\leq k\leq \left\lfloor \tfrac{n}{2}\right\rfloor .  \label{3.2}
	\end{equation}
\end{definition}

Some special cases of the incomplete $q$-Chebyshev polynomials of the first kind are provided in Table 2.

\begin{theorem}
	The incomplete $q$-Chebyshev Polynomials of the second kind satisfy
	\begin{equation}
	\:\mathcal{U}_{n+2}^{k+1}=(1+q^{n+2})x\:\mathcal{U}_{n+1}^{k+1}+q^{n+1}s\:\mathcal{U}_{n}^{k}\label{3.3}
	\end{equation}%
	for $0\leq k\leq \frac{n-1}{2}$.
\end{theorem}

\begin{proof}
	From Eq. (\ref{3.1}), we can write
	\begin{align*}
		(1+q^{n+2})x\:\mathcal{U}_{n+1}^{k+1}+q^{n+1}s\:\mathcal{U}_{n}^{k}=& \ (1+q^{n+2})x\sum\limits_{j=0}^{k+1}q^{j^{2}}\biggl[\displaystyle{n-j+1 \atop j}\biggr]_{q}\frac{(-q;q)_{n-j+1}}{(-q;q)_{j}}s^{j}x^{n+1-2j} \\ & \ +q^{n+1}s\sum_{j=0}^{k}q^{j^{2}}\biggl[\displaystyle{n-j \atop j }\biggr]_{q}\frac{(-q;q)_{n-j}}{(-q;q)_{j}}s^{j}x^{n-2j} \\
		=& \ \sum\limits_{j=0}^{k+1}q^{j^{2}}\left\{ (1+q^{n+2})\biggl[\displaystyle{n-j+1 \atop j }\biggr]_{q}+q^{n+1}q^{-2j+1}(1+q^{j})\biggl[\displaystyle{n-j+1 \atop j-1 }\biggr]_{q}\right\} \\ & \times \frac{(-q;q)_{n-j+1}}{(-q;q)_{j}}s^{j}x^{n-2j+2} \\
		=& \ \sum\limits_{j=0}^{k+1}q^{j^{2}}\left\{ \left( \biggl[\displaystyle{n-j+1 \atop j }\biggr]_{q}+q^{n-2j+2}\biggl[\displaystyle{n-j+1 \atop j-1 }\biggr]_{q}\right) \right. \\
		&\ \left. +q^{n-j+2}\left(q^{j}\biggl[\displaystyle{n-j+1 \atop j }\biggr]_{q}+\biggl[\displaystyle{n-j+1 \atop j-1 }\biggr]{j-1}_{q}\right) \right\} \frac{(-q;q)_{n-j+1}}{(-q;q)_{j}}s^{j}x^{n-2j+2}.
	\end{align*}
	Thus using Eq. (\ref{1.2a}) and Eq. (\ref{1.2b}), we get 
	\begin{align*}
		(1+q^{n+2})x\:\mathcal{U}_{n+1}^{k+1}+q^{n+1}s\:\mathcal{U}_{n}^{k}=& \ \sum\limits_{j=0}^{k+1}q^{j^{2}}\left(1+q^{n-j+2}\right) \biggl[\displaystyle{n-j+2 \atop j }\biggr]_{q}\frac{(-q;q)_{n-j+1}}{(-q;q)_{j}}s^{j}x^{n-2j+2} \\
		=& \ \sum\limits_{j=0}^{k+1}q^{j^{2}}\biggl[\displaystyle{n-j+2 \atop j }\biggr]_{q}\frac{(-q;q)_{n-j+2}}{(-q;q)_{j}}s^{j}x^{n-2j+2} \\
		=& \ \mathcal{U}_{n+2}^{k+1}.
	\end{align*}
\end{proof}

\begin{corollary}
	Incomplete $q$-Chebyshev Polynomials of the second kind satisfy the non-homogeneous recurrence relation
	\begin{equation}
	\mathcal{U}_{n+2}^{k} =(1+q^{n+2})x\:\mathcal{U}_{n+1}^{k}+q^{n+1}s\:\mathcal{U}_{n}^{k} -q^{n+1+k^{2}}\biggl[\displaystyle{n-k \atop k }\biggr]_{q}\frac{(-q;q)_{n-k}}{(-q;q)_{k}}s^{k+1}x^{n-2k}.  \label{3.4}
	\end{equation}		
\end{corollary}

\begin{theorem}
	For $0\leq k\leq \tfrac{n+1}{2},$ the following equality give a relationships between the incomplete $q$-Chebyshev polynomials of the first and second kind
	\begin{equation}
	\mathcal{T}_{n+2}^{k}=x\:\mathcal{U}_{n+1}^{k}+q^{n+1}s\:\mathcal{U}_{n}^{k-1}.\label{3.5}
	\end{equation}
\end{theorem}

\begin{proof}
	Using Eq. (\ref{3.1}), we obtain
	\begin{align*}
		\:\mathcal{U}_{n+1}^{k}+q^{n+1}s\:\mathcal{U}_{n}^{k-1} 	=\ &  x\sum\limits_{j=0}^{k}q^{j^{2}}\biggl[\displaystyle{n-j+1 \atop j }\biggr]_{q}\frac{(-q;q)_{n-j+1}}{(-q;q)_{j}}s^{j}x^{n+1-2j}+q^{n+1}s\sum\limits_{j=0}^{k-1}q^{j^{2}}\biggl[\displaystyle{n-j \atop j }\biggr]_{q}\frac{(-q;q)_{n-j}}{(-q;q)_{j}}s^{j}x^{n-2j} \\
		= & \ \sum\limits_{j=0}^{k}q^{j^{2}}\left\{ \biggl[\displaystyle{n-j+1 \atop j }\biggr]_{q}+q^{n+1-2j+1}(1+q^{j})\biggl[\displaystyle{n-j+1 \atop j-1 }\biggr]_{q}\right\}  \frac{(-q;q)_{n-j+1}}{(-q;q)_{j}}s^{j}x^{n-2j+2} 
	\end{align*}
	From Eq. (\ref{1.2b}) and Eq. (\ref{1.2d}), we get
	\begin{align*}
		xU_{n+1}^{k}+q^{n+1}sU_{n}^{k-1} = & \ \sum\limits_{j=0}^{k}q^{j^{2}}\frac{\left[ n+2\right] _{q}}{\left[ n-j+2\right] _{q}}\biggl[\displaystyle{n-j+2 \atop j }\biggr]_{q}\frac{(-q;q)_{n-j+1}}{(-q;q)_{j}}s^{j}x^{n-2j+2} \\
		=& \ \mathcal{T}_{n+2}^{k}.
	\end{align*}
\end{proof}

\begin{theorem}
	The incomplete $q$-Chebyshev polynomials of the first kind satisfy
	\begin{equation}
	\mathcal{T}_{n+2}^{k+1}=(1+q^{n+1})x\mathcal{T}_{n+1}^{k+1}+q^{n+1}s\mathcal{T}_{n}^{k}\label{3.6}
	\end{equation}
	for $0\leq k\leq \frac{n-1}{2}$. 
\end{theorem}

\begin{proof}
	By using Eq. (\ref{3.3}) and Eq. (\ref{3.5}), we get
	\begin{align*}
		\mathcal{T}_{n+2}^{k+1}=& \ x\:\mathcal{U}_{n+1}^{k+1}+q^{n+1}s\:\mathcal{U}_{n}^{k} \\
		=& \ (1+q^{n+1})x^{2}\:\mathcal{U}_{n}^{k+1}+q^{n}sx\:\mathcal{U}_{n-1}^{k}+q^{n+1}s(1+q^{n})x\:\mathcal{U}_{n-1}^{k}+q^{2n}s^{2}\:\mathcal{U}_{n-2}^{k-1} \\
		=&\ (1+q^{n+1})x\left\{ x\:\mathcal{U}_{n}^{k+1}+q^{n}s\:\mathcal{U}_{n-1}^{k}\right\}+q^{n+1}s\left\{ x\:\mathcal{U}_{n-1}^{k}+q^{n-1}s\:\mathcal{U}_{n-2}^{k-1}\right\}\\
		=& (1+q^{n+1})x\mathcal{T}_{n+1}^{k+1}+q^{n+1}s\mathcal{T}_{n}^{k}.
	\end{align*}
\end{proof}
\begin{corollary}
	Incomplete $q$-Chebyshev polynomials of the first kind satisfy the non-homogeneous recurrence relation 
	\begin{equation}
	\mathcal{T}_{n+2}^{k} =(1+q^{n+1})\mathcal{T}_{n+1}^{k}+q^{n+1}s\mathcal{T}_{n}^{k}-q^{n+1+k^{2}}\frac{\left[ n\right] _{q}}{\left[ n-k\right] _{q}}\biggl[\displaystyle{n-k \atop k }\biggr]_{q}\frac{(-q;q)_{n-k-1}}{(-q;q)_{k}}s^{k+1}x^{n-2k}.  \label{3.7}
	\end{equation}
\end{corollary}

\begin{theorem}
	For $0\leq k\leq \tfrac{n+1}{2}$, then
	\begin{equation}
	\mathcal{T}_{n+2}^{k}=x\:\mathcal{U}_{n+1}^{k}(x,q^{2}s,q)+qs\:\mathcal{U}_{n}^{k-1}(x,q^{2}s,q)   \label{3.8}
	\end{equation}
	holds.
\end{theorem}

\begin{proof}
	We obtain from Eq. (\ref{3.1})
	\begin{align*}
		x\ \mathcal{U}_{n+1}^{k}(x,q^{2}s,q)+qs\:\mathcal{U}_{n}^{k-1}(x,q^{2}s,q) 		=& \ x\sum\limits_{j=0}^{k}q^{j^{2}}\biggl[\displaystyle{n-j+1 \atop j }\biggr]_{q}\frac{(-q;q)_{n+1-j}}{(-q;q)_{j}}(q^{2}s)^{j}x^{n+1-2j}\\ & +qs\sum\limits_{j=0}^{k-1}q^{j^{2}}\biggl[\displaystyle{n-j \atop j }\biggr]_{q}\frac{(-q;q)_{n-j}}{(-q;q)_{j}}(q^{2}s)^{j}x^{n-2j} \\
		=& \ \sum\limits_{j=0}^{k}q^{j^{2}}\left\{ q^{2j}\biggl[\displaystyle{n-j+1 \atop j }\biggr]_{q}+(1+q^{j})\biggl[\displaystyle{n-j+1 \atop j-1 }\biggr]_{q}\right\} \frac{(-q;q)_{n+1-j}}{(-q;q)_{j}}s^{j}x^{n+2-2j} \\
		=& \ \sum\limits_{j=0}^{k}q^{j^{2}}\left\{ q^{j}\left\{ q^{j}\biggl[\displaystyle{n-j+1 \atop j }\biggr]_{q}+\biggl[\displaystyle{n-j+1 \atop j-1 }\biggr]_{q}\right\} +\biggl[\displaystyle{n-j+1 \atop j-1 }\biggr]_{q}\right\}\\ & \ \times  \frac{(-q;q)_{n+1-j}}{(-q;q)_{j}}s^{j}x^{n+2-2j}.
	\end{align*}
	By using Eq. (\ref{1.2a}), we have
	\begin{align*}
		x\:\mathcal{U}_{n+1}^{k}(x,q^{2}s,q)+qs\:\mathcal{U}_{n}^{k-1}(x,q^{2}s,q)=& \ \sum\limits_{j=0}^{k}q^{j^{2}}\left\{ q^{j}\biggl[\displaystyle{n-j+2 \atop j }\biggr]_{q}+\biggl[\displaystyle{n-j+1 \atop j-1 }\biggr]_{q}\right\} \frac{(-q;q)_{n+1-j}}{(-q;q)_{j}}s^{j}x^{n+2-2j}\\
		=& \ \sum\limits_{j=0}^{k}q^{j^{2}}\frac{\left[ n+2\right] _{q}}{\left[ n+2-j\right] _{q}}\biggl[\displaystyle{n-j+2 \atop j }\biggr]_{q}\frac{(-q;q)_{n+1-j}}{(-q;q)_{j}}s^{j}x^{n+2-2j} \\
		=& \ \mathcal{T}_{n+2}^{k}.
	\end{align*}
\end{proof}
\begin{theorem}
	We have
	\begin{equation}
	(1+q^{n+2})\mathcal{T}_{n+2}^{k}=\:\mathcal{U}_{n+2}^{k}+q^{2n+3}s\:\mathcal{U}_{n}^{k-1},\ \ \ \ \ 0\leq k\leq \left\lfloor \tfrac{n}{2}\right\rfloor .  \label{3.9}
	\end{equation}
\end{theorem}

\begin{proof}
	From Eq. (\ref{3.3}) and Eq. (\ref{3.1}), we get
	\begin{align*}
		\:\mathcal{U}_{n+2}^{k} +q^{2n+3}s\:\mathcal{U}_{n}^{k-1} 		=&\ \left\{(1+q^{n+2})x\:\mathcal{U}_{n+1}^{k}+q^{n+1}s\:\mathcal{U}_{n}^{k-1}\right\}+q^{2n+3}s\:\mathcal{U}_{n}^{k-1} \\
		=&\ \sum\limits_{j=0}^{k}q^{j^{2}}\left\{ \biggl[\displaystyle{n-j+1 \atop j }\biggr]_{q}+q^{n+1-2j+1}(1+q^{j})\biggl[\displaystyle{n-j+1 \atop j-1 }\biggr]_{q}\right\}  \frac{(-q;q)_{n+1-j}}{(-q;q)_{j}}s^{j}x^{n+2-2j} \\
		&\ +q^{n+2}\sum\limits_{j=0}^{k}q^{j^{2}}\left\{ \biggl[\displaystyle{n-j+1 \atop j }\biggr]_{q}+q^{n+1-2j+1}(1+q^{j})\biggl[\displaystyle{n-j+1 \atop j-1 }\biggr]_{q}\right\}  \frac{(-q;q)_{n+1-j}}{(-q;q)_{j}}s^{j}x^{n+2-2j} \\
	\end{align*}
	We get the following result from Eq. (\ref{1.2b}) and Eq. (\ref{1.2d})
	\begin{align*}
		\mathcal{U}_{n+2}^{k}+q^{2n+3}s\:\mathcal{U}_{n}^{k-1}		=& \ \sum\limits_{j=0}^{k}q^{j^{2}}\frac{\left[ n+2\right] _{q}}{\left[ n+2-j\right] _{q}}\biggl[\displaystyle{n-j+2 \atop j }\biggr]_{q}\frac{(-q;q)_{n+1-j}}{(-q;q)_{j}}s^{j}x^{n+2-2j}\\
		&\ +q^{n+2}\sum\limits_{j=0}^{k}q^{j^{2}}\frac{\left[ n+2\right]_{q}}{\left[ n+2-j\right] _{q}}\biggl[\displaystyle{n-j+2 \atop j }\biggr]_{q}\frac{(-q;q)_{n+1-j}}{(-q;q)_{j}}s^{j}x^{n+2-2j} \\
		=&\ \mathcal{T}_{n+2}^{k}+q^{n+2}\mathcal{T}_{n+2}^{k}.
	\end{align*}
\end{proof}
\begin{lemma} \label{lemma1}
	We have
	\begin{equation}
	\frac{d\,\:\mathcal{U}_{n}}{dx}=nx^{-1}\:\mathcal{U}_{n}-2x^{-1}\sum\limits_{j=0}^{\lfloor \frac{n}{2}\rfloor }jq^{j^{2}}\biggl[\displaystyle{n-j \atop j }\biggr]_{q}\frac{(-q;q)_{n-j}}{(-q;q)_{j}}s^{j}x^{n-2j}  \label{3.9a}
	\end{equation}
	and
	\begin{equation}
	\frac{d\,\mathcal{T}_{n}}{dx}=nx^{-1}\mathcal{T}_{n}-2x^{-1}\sum\limits_{j=0}^{\lfloor \frac{n}{2}\rfloor }jq^{j^{2}}\frac{\left[ n\right] _{q}}{\left[ n-j\right] _{q}}\biggl[\displaystyle{n-j \atop j }\biggr]_{q}\frac{(-q;q)_{n-j-1}}{(-q;q)_{j}}s^{j}x^{n-2j}.  \label{3.9b}
	\end{equation}
\end{lemma}

\begin{proof}
	By using Eq. (\ref{1.5}), we have
	\begin{eqnarray*}
		\frac{d\,\:\mathcal{U}_{n}}{dx} &=&\frac{d}{dx}\biggl\{\sum\limits_{j=0}^{\left\lfloor \frac{n}{2}\right\rfloor }q^{j^{2}}\biggl[\displaystyle{n-j \atop j}\biggr]_{q}\frac{(-q;q)_{n-j}}{(-q;q)_{j}}s^{j}x^{n-2j}\biggr\} \\&=&\sum\limits_{j=0}^{\left\lfloor \frac{n}{2}\right\rfloor }q^{j^{2}}(n-2j)\biggl[\displaystyle{n-j \atop j }\biggr]_{q}\frac{(-q;q)_{n-j}}{(-q;q)_{j}}s^{j}x^{n-2j-1} \\&=&nx^{-1}\:\mathcal{U}_{n}-2\sum\limits_{j=0}^{\left\lfloor \frac{n}{2}\right\rfloor }jq^{j^{2}}\biggl[\displaystyle{n-j \atop j }\biggr]_{q}\frac{(-q;q)_{n-j}}{(-q;q)_{j}}s^{j}x^{n-2j-1}.
	\end{eqnarray*}
	Similarly, from Eq. (\ref{1.6}), we get Eq. (\ref{3.9b}).
\end{proof}
Using Lemma \ref{lemma1} , we can prove the following theorem.
\begin{theorem}
	We have
	\begin{equation}
	\sum\limits_{k=0}^{\left\lfloor \frac{n}{2}\right\rfloor}\:\mathcal{U}_{n}^{k}=\left( \left\lfloor \frac{n}{2}\right\rfloor -\frac{n}{2}+1\right) \:\mathcal{U}_{n}+\frac{x}{2}\frac{d\:\mathcal{U}_{n}}{dx}.  \label{3.10}
	\end{equation}
\end{theorem}

\begin{proof}
	From Eq. (\ref{3.1}), we have
	\begin{align*}
		\sum\limits_{k=0}^{\lfloor \frac{n}{2}\rfloor }\:\mathcal{U}_{n}^{k}=&\ \mathcal{U}_{n}^{0}+\:\mathcal{U}_{n}^{1}+\cdots+\:\mathcal{U}_{n}^{\left\lfloor \frac{n}{2}\right\rfloor } \\
		=& \ \left( q^{0}\biggl[\displaystyle{n \atop 0 }\biggr]_{q}\frac{(-q;q)_{n}}{(-q;q)_{0}}x^{n}\right) +\left( q^{0}\biggl[\displaystyle{n \atop 0}\biggr]_{q}\frac{(-q;q)_{n}}{(-q;q)_{0}}x^{n}+q
		\biggl[\displaystyle{n-1 \atop 1 }\biggr]_{q}\frac{(-q;q)_{n-1}}{(-q;q)_{1}}sx^{n-2}\right)+ \cdots \\
		&\ +\biggl(q^{0}\biggl[\displaystyle{n \atop 0 }\biggr]_{q}\frac{(-q;q)_{n}}{(-q;q)_{0}}x^{n}+q\biggl[\displaystyle{n-1 \atop 1 }\biggr]_{q}\frac{(-q;q)_{n-1}}{(-q;q)_{1}}sx^{n-2}+ \cdots +q^{\lfloor \frac{n}{2}\rfloor }
		\biggl[\displaystyle{n-\lfloor \frac{n}{2}\rfloor \atop \lfloor \frac{n}{2}\rfloor }\biggr]_{q}\frac{(-q;q)_{n-\lfloor \frac{n}{2}\rfloor }}{(-q;q)_{\lfloor \frac{n}{2}\rfloor }}s^{\lfloor \frac{n}{2}\rfloor }x^{n-2\lfloor \frac{n}{2}\rfloor }\biggr) \\
		=&\ \biggl(\left\lfloor \frac{n}{2}\right\rfloor +1\biggr)\left( q^{0}
		\biggl[\displaystyle{n \atop 0 }\biggr]_{q}\frac{(-q;q)_{n}}{(-q;q)_{0}}x^{n}\right) +\biggl(\left\lfloor \frac{n}{2}\right\rfloor +1-1\biggr)\left( q
		\biggl[\displaystyle{n-1 \atop 1}\biggr]_{q}\frac{(-q;q)_{n-1}}{(-q;q)_{1}}sx^{n-2}\right)+\cdots \\
		& \ +\biggl(\left\lfloor \frac{n}{2}\right\rfloor +1-\left\lfloor \frac{n}{2}\right\rfloor \biggr)\left( q^{\lfloor \frac{n}{2}\rfloor }
		\biggl[\displaystyle{n-\lfloor\frac{n}{2}\rfloor \atop \lfloor \frac{n}{2}\rfloor }\biggr]_{q}\frac{(-q;q)_{n-\lfloor \frac{n}{2}\rfloor }}{(-q;q)_{\lfloor \frac{n}{2}\rfloor }}s^{\lfloor \frac{n}{2}\rfloor }x^{n-2\lfloor \frac{n}{2}\rfloor }\right) \\
		=& \ \sum\limits_{j=0}^{\lfloor \frac{n}{2}\rfloor }\left( \left\lfloor \frac{n}{2}\right\rfloor +1-j\right) q^{J^{2}}
		\biggl[\displaystyle{n-j \atop j }\biggr]_{q}\frac{(-q;q)_{n-j}}{(-q;q)_{j}}s^{j}x^{n-2j} \\
		=&\ \biggl(\left\lfloor \frac{n}{2}\right\rfloor +1\biggr)\:\mathcal{U}_{n}-\sum\limits_{j=0}^{\lfloor \frac{n}{2}\rfloor }jq^{j^{2}}
		\biggl[\displaystyle{n-j \atop j }\biggr]_{q}\frac{(-q;q)_{n-j}}{(-q;q)_{j}}s^{j}x^{n-2j}.
	\end{align*}
	Then by using Lemma \ref{lemma1}, we get
	\begin{equation*}
	\sum\limits_{k=0}^{\lfloor \frac{n}{2}\rfloor }\:\mathcal{U}_{n}^{k}=\left(\left\lfloor \frac{n}{2}\right\rfloor -\frac{n}{2}+1\right) \:\mathcal{U}_{n}+\frac{x}{2}\frac{d\,\:\mathcal{U}_{n}}{dx}.
	\end{equation*}
\end{proof}

\begin{theorem}
	We have
	\begin{equation}
	\sum\limits_{k=0}^{\left[ \frac{n}{2}\right] }\mathcal{T}_{n}^{k}=\left(\left\lfloor \frac{n}{2}\right\rfloor -\frac{n}{2}+1\right) \mathcal{T}_{n}+\frac{x}{2}\frac{d\,\mathcal{T}_{n}}{dx}.  \label{3.11}
	\end{equation}
\end{theorem}

\begin{proof}
	We have from Eq. (\ref{3.2})
	\begin{align*}
		\sum\limits_{k=0}^{\lfloor \frac{n}{2}\rfloor }\mathcal{T}_{n}^{k}=&\ \mathcal{T}_{n}^{0}+\mathcal{T}_{n}^{1}+\cdots+\mathcal{T}_{n}^{\left\lfloor \frac{n}{2}\right\rfloor } \\
		=& \ (q^{0}
		\biggl[\displaystyle{n \atop 0 }\biggr]_{q}\frac{(-q;q)_{n-1}}{(-q;q)_{0}}x^{n})+(q^{0}\biggl[\displaystyle{n \atop 0 }\biggr]_{q}\frac{(-q;q)_{n-1}}{(-q;q)_{0}}x^{n}+q\frac{\left[ n\right] _{q}}{\left[ n-1\right] _{q}}\biggl[\displaystyle{n-1 \atop 1 }\biggr]_{q}\frac{(-q;q)_{n-2}}{(-q;q)_{1}}sx^{n-2})+\cdots\\
		& \ +\left( q^{0}\biggl[\displaystyle{n \atop 0 }\biggr]_{q}\frac{(-q;q)_{n-1}}{(-q;q)_{0}}x^{n}+q\frac{\left[ n\right] _{q}}{\left[ n-1\right] _{q}}\biggl[\displaystyle{n-1 \atop 1 }\biggr]				_{q}\frac{(-q;q)_{n-2}}{(-q;q)_{1}}sx^{n-2}\right.+\cdots \\
		&\ \left. +q^{^{\lfloor \frac{n}{2}\rfloor ^{2}}}\frac{\left[ n\right]_{q}}{\left[ n-\lfloor \frac{n}{2}\rfloor \right] _{q}}
		\biggl[\displaystyle{n-\lfloor \frac{n}{2}\rfloor \atop \lfloor \frac{n}{2}\rfloor }\biggr]_{q}\frac{(-q;q)_{n-\lfloor \frac{n}{2}\rfloor -1}}{(-q;q)_{\lfloor \frac{n}{2}\rfloor}}s^{\lfloor \frac{n}{2}\rfloor }x^{n-2\lfloor \frac{n}{2}\rfloor }\right) \\
		=&\ \sum\limits_{j=0}^{\lfloor \frac{n}{2}\rfloor }\left( \left\lfloor \frac{n}{2}\right\rfloor +1-j\right) q^{j^{2}}\frac{\left[ n\right] _{q}}{\left[n-j\right] _{q}}\biggl[\displaystyle{n-j \atop j }\biggr]_{q}\frac{(-q;q)_{n-j-1}}{(-q;q)_{j}}s^{j}x^{n-2j} \\
		=& \ \left( \left\lfloor \frac{n}{2}\right\rfloor +1\right)\mathcal{T}_{n}-\sum\limits_{j=0}^{\lfloor \frac{n}{2}\rfloor }jq^{j^{2}}\frac{\left[ n\right] _{q}}{\left[ n-j\right] _{q}}\biggl[\displaystyle{n-j \atop j }\biggr]_{q}\frac{(-q;q)_{n-j-1}}{(-q;q)_{j}}s^{j}x^{n-2j}.
	\end{align*}
	Lemma \ref{lemma1} implies that
	\begin{equation*}
	\sum\limits_{k=0}^{\lfloor \frac{n}{2}\rfloor }\mathcal{T}_{n}^{k}=\left(\left\lfloor \frac{n}{2}\right\rfloor -\frac{n}{2}+1\right) \mathcal{T}_{n}+\frac{x}{2}\frac{d\,\mathcal{T}_{n}}{dx}.
	\end{equation*}
\end{proof}

\section{Tables and Graphs of Incomplete $q$-Chebyshev polynomials}

In this section, we display the graphs of the $q$-Chebyshev polynomials and incomplete $q$-Chebyshev polynomials. Also we give the tables of some special cases of incomplete $q$-Chebyshev polynomials and numbers. 
\begin{table}[h!]
	\centering
	\caption{Some special cases of the incomplete $q$-Chebyshev polynomials of the second kind}
\begin{equation*}
		\begin{tabular}{|l|c|l|l|l|}
			\hline
			$x$ & $s$ & $q$ & $\:\mathcal{U}_{n}^{k}(x,s,q)$ & Incomplete $q$-Chebyshev polynomials of the second kind \\ \hline\hline
			$x$ & $-1$ & $1$ & $U_{n}^{k}(x)$ & Incomplete Chebyshev polynomials of the second kind \\ \hline
			$\frac{x}{2}$ & $y$ & $1$ & $F_{n+1}^{(k)}(x,y)$ & Incomplete bivariate Fibonacci polynomials \\ \hline
			$\frac{x}{2}$ & $1$ & $1$ & $F_{n+1}^{(k)}(x)$ & Incomplete Fibonacci polynomials \\ \hline
			$\frac{1}{2}$ & $1$ & $1$ & $F_{n+1}^{(k)}$ & Incomplete Fibonacci numbers \\ \hline
			$x$ & $1$ & $1$ & $P_{n+1}^{(k)}(x)$ & Incomplete Pell polynomials \\ \hline
			$1$ & $1$ & $1$ & $P_{n+1}^{(k)}$ & Incomplete Pell numbers \\ \hline
			$\frac{1}{2}$ & $2y$ & $1$ & $J_{n+1}^{(k)}(x)$ & Incomplete Jacobsthal polynomials \\ \hline
			$\frac{1}{2}$ & $2$ & $1$ & $J_{n+1}^{(k)}$ & Incomplete Jacobsthal numbers\\ \hline
		\end{tabular}
	\end{equation*}	
\end{table}
\begin{table}[h!]
	\centering
	\caption{Some special cases of the incomplete $q$-Chebyshev polynomials of the first kind}
	
	\begin{equation*}
		\begin{tabular}{|l|c|l|l|l|}
			\hline
			$x$ & $s$ & $q$ & $\mathcal{T}_{n}^{k}(x,s,q)$ & Incomplete $q$-Chebyshev polynomials of the first kind \\ \hline\hline
			$x$ & $-1$ & $1$ & $T_{n}^{k}(x)$ & Incomplete Chebyshev polynomials of the first kind \\ \hline
			$\frac{x}{2}$ & $y$ & $1$ & $\frac{1}{2}L_{n}^{(k)}(x,y)$ & Incomplete
			Bivariate Lucas polynomials \\ \hline
			$\frac{x}{2}$ & $1$ & $1$ & $\frac{1}{2}L_{n}^{(k)}(x)$ & Incomplete Lucas
			polynomials \\ \hline
			$\frac{1}{2}$ & $1$ & $1$ & $\frac{1}{2}L_{n}^{(k)}$ & Incomplete Lucas
			numbers \\ \hline
			$x$ & $1$ & $1$ & $\frac{1}{2}Q_{n}^{(k)}(x)$ & Incomplete Pell-Lucas
			polynomials \\ \hline
			$1$ & $1$ & $1$ & $\frac{1}{2}Q_{n}^{(k)}$ & Incomplete Pell-Lucas numbers
			\\ \hline
			$\frac{1}{2}$ & $2y$ & $1$ & $\frac{1}{2}~j_{n}^{(k)}(y)$ & Incomplete
			Jacobsthal-Lucas polynomials \\ \hline
			$\frac{1}{2}$ & $2$ & $1$ & $\frac{1}{2}~j_{n}^{(k)}$ & Incomplete
			Jacobsthal-Lucas numbers \\ \hline
		\end{tabular}
	\end{equation*}
\end{table}

The numerical results for the incomplete Chebyshev numbers, first and second kind, incomplete Fibonacci, incomplete Pell, and incomplete Jacobsthal numbers are displayed in Table 4-5.
The numerical results for the incomplete Lucas, incomplete Pell-Lucas, and incomplete Jacobsthal-Lucas numbers are displayed in Table 6.

\begin{table}[h!]
	\centering
	\caption{Incomplete Chebyshev numbers of the first and second kind }
	
	\begin{equation*}
	\begin{tabular}{c|rrrrrc|rrrrr}
	&  &\multicolumn{4}{l}{$\mathcal{T}_{n}^{k}(1,-1,1)$}  & &&\multicolumn{4}{l}{$\:\mathcal{U}_{n}^{k}(1,-1,1)$} \\[10pt]
	$n/k$ & $0$ & $1$ & $2$ & $3$ & $4$ && $0$ & $1$ & $2$ & $3$ & $4$ \\ \hline\hline
	$1$ & $1$ &  &  &  & & & $2$ &  &  &  &   \\ 
	$2$ & $2$ & $1$ &  &  & & 	& $4$ & $3$ &  &  &   \\ 
	$3$ & $4$ & $1$ &  &  & &  & $8$ & $4$ &  &  &    \\ 
	$4$ & $8$ & $0$ & $1$ &  & && $16$ & $4$ & $5$ &  &     \\ 
	$5$ & $16$ & $-4$ & $1$ &  & &	& $32$ & $0$ & $6$ &  &    \\ 
	$6$ & $32$ & $-16$ & $2$ & $1$ & && $64$ & $-16$ & $8$ & $7$ &     \\ 
	$7$ & $64$ & $-48$ & $8$ & $1$ & && $128$ & $-64$ & $16$ & $8$ &     \\ 
	$8$ & $128$ & $-128$ & $32$ & $0$ & $1$ && $256$ & $-192$ & $48$ & $8$ & $9$    \\ 
	$9$ & $256$ & $-320$ & $112$ & $-8$ & $1$  && $512$ & $-512$ & $160$ & $0$ & $10$  
	\end{tabular}
	\end{equation*}
\end{table}

\bigskip 
\begin{table}[h!]
	\centering
	\caption{Incomplete Fibonacci numbers, incomplete Pell numbers and incomplete Jacobsthal numbers }
	
	\begin{equation*}
	\begin{tabular}{c|lllllc|lllllc|lllll}
	&  & \multicolumn{4}{l}{$\:\mathcal{U}_{n}^{k}(\frac{1}{2},1,1)=F_{n+1}^{(k)}$}& &  & \multicolumn{4}{l}{$\:\mathcal{U}_{n}^{k}(1,1,1)=P_{n+1}^{(k)}$} & 	&  & \multicolumn{4}{l}{$\:\mathcal{U}_{n}^{k}(\frac{1}{2},2,1)=J_{n+1}^{(k)}$}\\[10pt]
	$n/k$ & $0$ & $1$ & $2$ & $3$ & $4$ & & $0$ & $1$ & $2$ & $3$ & $4$ & & $0$ & $1$ & $2$ & $3$ & $4$ \\ \hline\hline
	$1$ & $1$ &  &  &  & & & $2$ &  &  &  & &	& $1$ &  &  &  &   \\ 
	$2$ & $1$ & $2$ &  &  & & 	& $4$ & $5$ &  &  & && $1$ & $3$ &  &  & \\ 
	$3$ & $1$ & $3$ &  &  & & 	& $8$ & $12$ &  &  & & & $1$ & $5$ &  &  &   \\ 
	$4$ & $1$ & $4$ & $5$ &  & & & $16$ & $28$ & $29$ &  & & & $1$ & $7$ & $11$ &  &   \\ 
	$5$ & $1$ & $5$ & $8$ &  & & & $32$ & $64$ & $70$ &  & & & $1$ & $9$ & $21$ &  &   \\ 
	$6$ & $1$ & $6$ & $12$ & $13$ & && $64$ & $144$ & $168$ & $169$ & & & $1$ & $114$ & $35$ & $43$ & \\ 
	$7$ & $1$ & $7$ & $17$ & $21$ & & & $128$ & $320$ & $400$ & $408$ & & & $1$ & $13$ & $53$ & $85$ & \\ 
	$8$ & $1$ & $8$ & $23$ & $33$ & $34$ && $256$ & $704$ & $944$ & $984$ & $985$& & $1$ & $15$ & $75$ & $155$ & $171$ \\ 
	$9$ & $1$ & $9$ & $30$ & $50$ & $55$ &	& $512$ & $1536$ & $2208$ & $2368$ & $2378$ & & $1$ & $17$ & $101$ & $261$ & $341$
	\end{tabular}
	\end{equation*}
\end{table}

\bigskip 
\begin{table}[h!]
	\centering
	\caption{Incomplete Lucas numbers, incomplete Pell-Lucas numbers and incomplete Jacobsthal-Lucas numbers }
	\begin{equation*}
	\begin{tabular}{c|lllllc|lllllc|lllll}
	&  & \multicolumn{4}{l}{$2\mathcal{T}_{n}^{k}(\frac{1}{2},1,1)=L_{n}^{(k)}$}& &  & \multicolumn{4}{l}{$2\mathcal{T}_{n}^{k}(1,1,1)=Q_{n}^{(k)}$}& &  & \multicolumn{4}{l}{$2\mathcal{T}_{n}^{k}(\frac{1}{2},2,1)=j_{n}^{(k)}$} \\[10pt]
	$n/k$ & $0$ & $1$ & $2$ & $3$ & $4$ & & $0$ & $1$ & $2$ & $3$ & $4$   & & $0$ & $1$ & $2$ & $3$ & $4$   \\ \hline\hline
	$1$ & $1$ &  &  &  & & & $2$ &  &  &  &    &  & $1$ &  &  &  &   \\ 
	$2$ & $1$ & $3$ &  &  & &   & $4$ & $6$ &  &  &  & & $1$ & $5$ &  &  &   \\ 
	$3$ & $1$ & $4$ &  &  &  &    & & $8$ & $14$ &  &  & 	& $1$ & $7$ &  &  &  \\ 
	$4$ & $1$ & $5$ & $7$ &  & &  & $16$ & $32$ & $34$ &  &   & & $1$ & $9$ & $17$ &  &    \\ 
	$5$ & $1$ & $6$ & $11$ &  &  &  & $32$ & $72$ & $82$ &  &  & & $1$ & $11$ & $31$ &  &   \\ 
	$6$ & $1$ & $7$ & $16$ & $18$ & & & $64$ & $160$ & $196$ & $198$ &   &  & $1$ & $13$ & $49$ & $65$ &  \\ 
	$7$ & $1$ & $8$ & $22$ & $29$ &  &  & $128$ & $352$ & $464$ & $478$ &  & & $1$ & $15$ & $71$ & $127$ &  \\ 
	$8$ & $1$ & $9$ & $29$ & $45$ & $47$ &  	& $256$ & $768$ & $1088$ & $1152$ & $1154$   & & $1$ & $17$ & $97$ & $225$ & $257$  \\ 
	$9$ & $1$ & $10$ & $37$ & $67$ & $76$& & $512$ & $1664$ & $2528$ & $2768$ & $2786$    &    	& $1$ & $19$ & $127$ & $367$ & $511$
	\end{tabular}
	\end{equation*}
\end{table}

In Figures 1, 2 the graphs of the $q$-Chebyshev polynomials of first and second kind for $s=-1$,\ $q=-0.5,\ 0.5,\ 0.9, \ 0.9999$,  $n=0,1,2,3,4,5$ and $-1\leq x \leq 1$ are shown.

\begin{center}
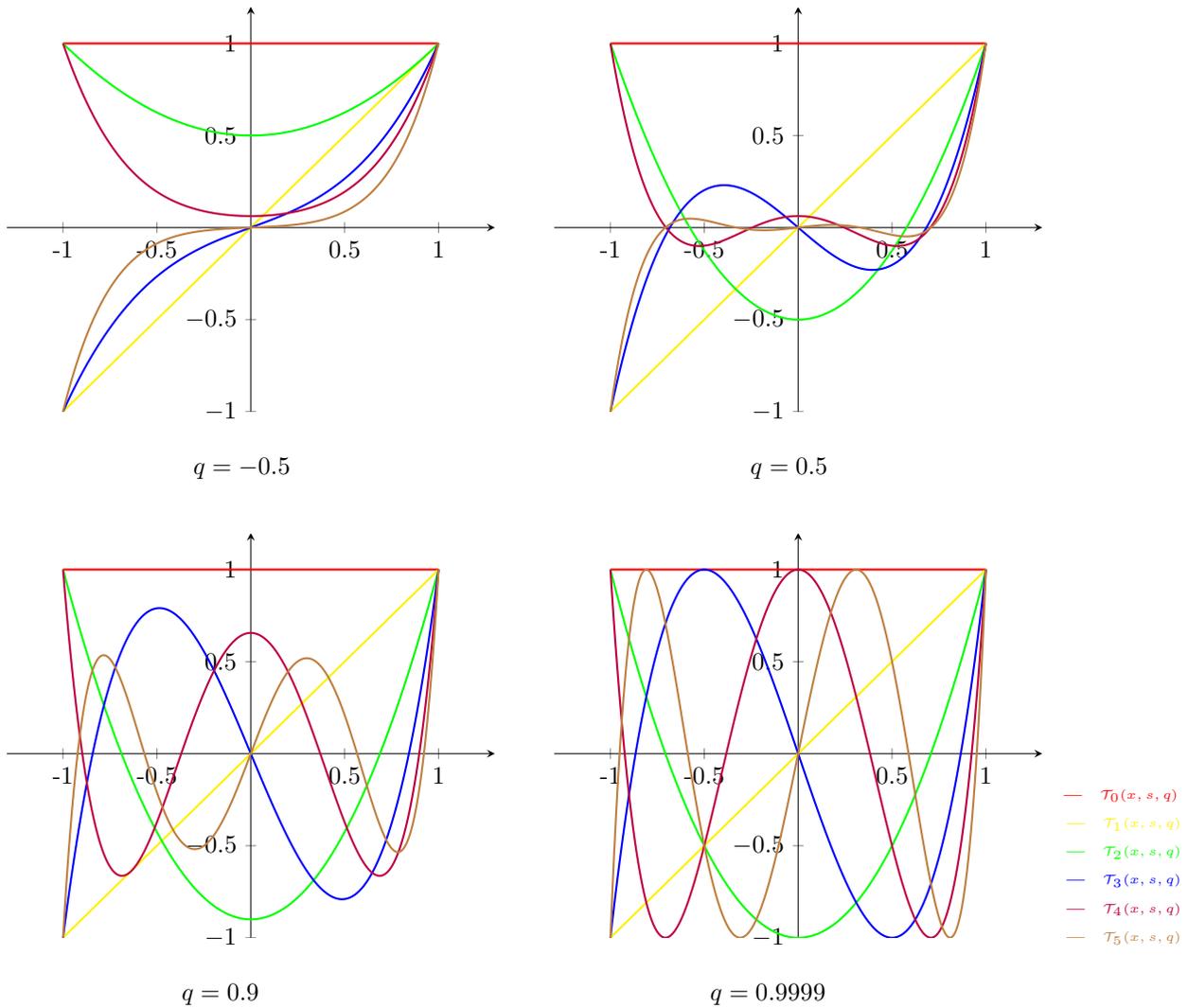
\begin{figure}[h!]

	\begin{tikzpicture}[scale=1.0]
	\begin{axis}[
	axis lines = center,
	ytick={-1,-0.5,0.5,1},
	xtick={-1,-0.5,0,0.5,1},
	xticklabels={-1,-0.5,0,0.5,1},
	]
	\addplot [
	domain=-1.299:1.3,
	range=-1.6..1.6, 
	samples=300, 
	color=white,
	]
	{1.2};
	\addplot [
	domain=-1:1, 
	range=-1.6..1.6,
	samples=300, 
	color=red, thick,
	]
	{1};
	\addplot [
	domain=-1:1,
	range=-1.6..1.6, 
	samples=300, 
	color=yellow, thick,
	]
	{x};
	\addplot [
	domain=-1:1,
	range=-1.6..1.6, 
	samples=300, 
	color=green, thick,
	]
	{.5*x^2+.5};
	\addplot [
	domain=-1:1,
	range=-1.6..1.6, 
	samples=299, 
	color=blue, thick,
	]
	{.625*x^3+.375*x};
	\addplot [
	domain=-1:1,
	range=-1.6..1.6, 
	samples=299, 
	color=purple, thick,
	]
	{.546875*x^4+.390625*x^2+.0625};     
	\addplot [
	domain=-1:1,
	range=-1.6..1.6, 
	samples=299, 
	color=brown, thick,
	]
	{.5810546875*x^5+.3759765625*x^3+.04296875*x};  
	\end{axis}
	\draw[very thick] (3.3,-0.8) node {  $q=-0.5 $};
	\end{tikzpicture}
	\qquad
	\begin{tikzpicture}[scale=1.0]
	\begin{axis}[
	axis lines = center,
	ytick={-1,-0.5,0.5,1},
	xtick={-1,-0.5,0,0.5,1},
	xticklabels={-1,-0.5,0,0.5,1},
	]
	\addplot [
	domain=-1.299:1.3,
	range=-1.6..1.6, 
	samples=300, 
	color=white, thick,
	]
	{1.2};
	\addplot [
	domain=-1:1,
	range=-1.6..1.6,
	samples=300, 
	color=red, thick,
	]
	{1};
	\addplot [
	domain=-1:1,
	range=-1.6..1.6, 
	samples=300, 
	color=yellow, thick,
	]
	{x};
	\addplot [
	domain=-1:1,
	range=-1.6..1.6, 
	samples=300, 
	color=green, thick,
	]
	{1.5*x^2-.5};
	\addplot [
	domain=-1:1,
	range=-1.6..1.6, 
	samples=299, 
	color=blue, thick,
	]
	{1.875*x^3-.875*x};
	\addplot [
	domain=-1:1,
	range=-1.6..1.6, 
	samples=299, 
	color=purple, thick,
	]
	{2.109375*x^4-1.171875*x^2+.0625};     
	\addplot [
	domain=-1:1,
	range=-1.6..1.6, 
	samples=299, 
	color=brown, thick,
	]
	{2.241210938*x^5-1.362304688*x^3+.12109375*x};  
	\end{axis}
	\draw[very thick] (3.3,-0.8) node {  $q=0.5 $};
	\end{tikzpicture}
	\\
	\vskip 5pt
	\begin{tikzpicture}[scale=1.0]
	\begin{axis}[
	axis lines = center,
	ytick={-1,-0.5,0.5,1},
	xtick={-1,-0.5,0,0.5,1},
	xticklabels={-1,-0.5,0,0.5,1},
	]
	\addplot [
	domain=-1.299:1.3,
	range=-1.6..1.6, 
	samples=300, 
	color=white,  thick,
	]
	{1.2};
	\addplot [
	domain=-1:1,
	range=-1.6..1.6,
	samples=300, 
	color=red, thick,
	]
	{1};
	\addplot [
	domain=-1:1,
	range=-1.6..1.6, 
	samples=300, 
	color=yellow, thick,
	]
	{x};
	\addplot [
	domain=-1:1,
	range=-1.6..1.6, 
	samples=300, 
	color=green, thick,
	]
	{1.9*x^2-.9};
	\addplot [
	domain=-1:1,
	range=-1.6..1.6, 
	samples=299, 
	color=blue, thick,
	]
	{3.439*x^3-2.439*x};
	\addplot [
	domain=-1:1,
	range=-1.6..1.6, 
	samples=299, 
	color=purple, thick,
	]
	{5.946031*x^4-5.602131*x^2+.6561};     
	\addplot [
	domain=-1:1,
	range=-1.6..1.6, 
	samples=299, 
	color=brown, thick,
	]
	{9.847221939*x^5-11.53401705*x^3+2.68679511*x};  
	\end{axis}
	\draw[very thick] (3,-0.8) node {  $q=0.9 $};
	\end{tikzpicture}
	\qquad
	\begin{tikzpicture}[scale=1.0]
	\begin{axis}[
	axis lines = center,
	ytick={-1,-0.5,0.5,1},
	xtick={-1,-0.5,0,0.5,1},
	xticklabels={-1,-0.5,0,0.5,1},
	]
	\addplot [
	domain=-1.299:1.3, 
	samples=300, 
	color=white, thick,
	]
	{1.2};
	\addplot [
	domain=-1:1, 
	range=-1.6..1.6,
	samples=300, 
	color=red, thick,
	]
	{1};
	\addplot [
	domain=-1:1,
	range=-1.6..1.6, 
	samples=300, 
	color=yellow, thick,
	]
	{x};
	\addplot [
	domain=-1:1,
	samples=300, 
	color=green, thick,
	]
	{1.9999*x^2-.9999};
	\addplot [
	domain=-1:1, 
	range=-1.6..1.6,
	samples=299, 
	color=blue, thick,
	]
	{3.999400040*x^3-2.999400040*x};
	\addplot [
	domain=-1:1,
	range=-1.6..1.6, 
	samples=299, 
	color=purple, thick,
	]
	{7.997600380*x^4-7.997200440*x^2+.9996000600};     
	\addplot [
	domain=-1:1, 
	samples=299, 
	color=brown, thick,
	]
	{15.99200220*x^5-19.98900300*x^3+4.997000800*x};  
	\end{axis}
	\draw[very thick] (8,2) node {\tiny {\color{red}{\textbf{---}}  \ \ $\mathcal{T}_{0}(x,s,q) \ $}};
	\draw[very thick] (8,1.6) node {\tiny {\color{yellow}{\textbf{---}}  \ \ $\mathcal{T}_{1}(x,s,q) $}};
	\draw[very thick] (8,1.2) node {\tiny {\color{green}{\textbf{---}}  \ \ $\mathcal{T}_{2}(x,s,q) $}};  
	\draw[very thick] (8,0.8) node {\tiny {\color{blue}{\textbf{---}} \ \ $\mathcal{T}_{3}(x,s,q) $}};
	\draw[very thick] (8,0.4) node {\tiny {\color{purple}{\textbf{---}} \ \ $\mathcal{T}_{4}(x,s,q) $}};
	\draw[very thick] (8,0.0) node {\tiny {\color{brown}{\textbf{---}} \ \ $\mathcal{T}_{5}(x,s,q) $}};
	\draw[very thick] (3,-0.8) node {  $q=0.9999 $};
	\end{tikzpicture}
	\caption{Graphs of $\mathcal{T}_{n}(x,s,q)$  for \ $s=-1$,\ $q=-0.5,0.5,0.9, 0.9999$, \ $n=0,1,2,3,4,5$ }
	\vskip 5pt
\end{figure}
\end{center}
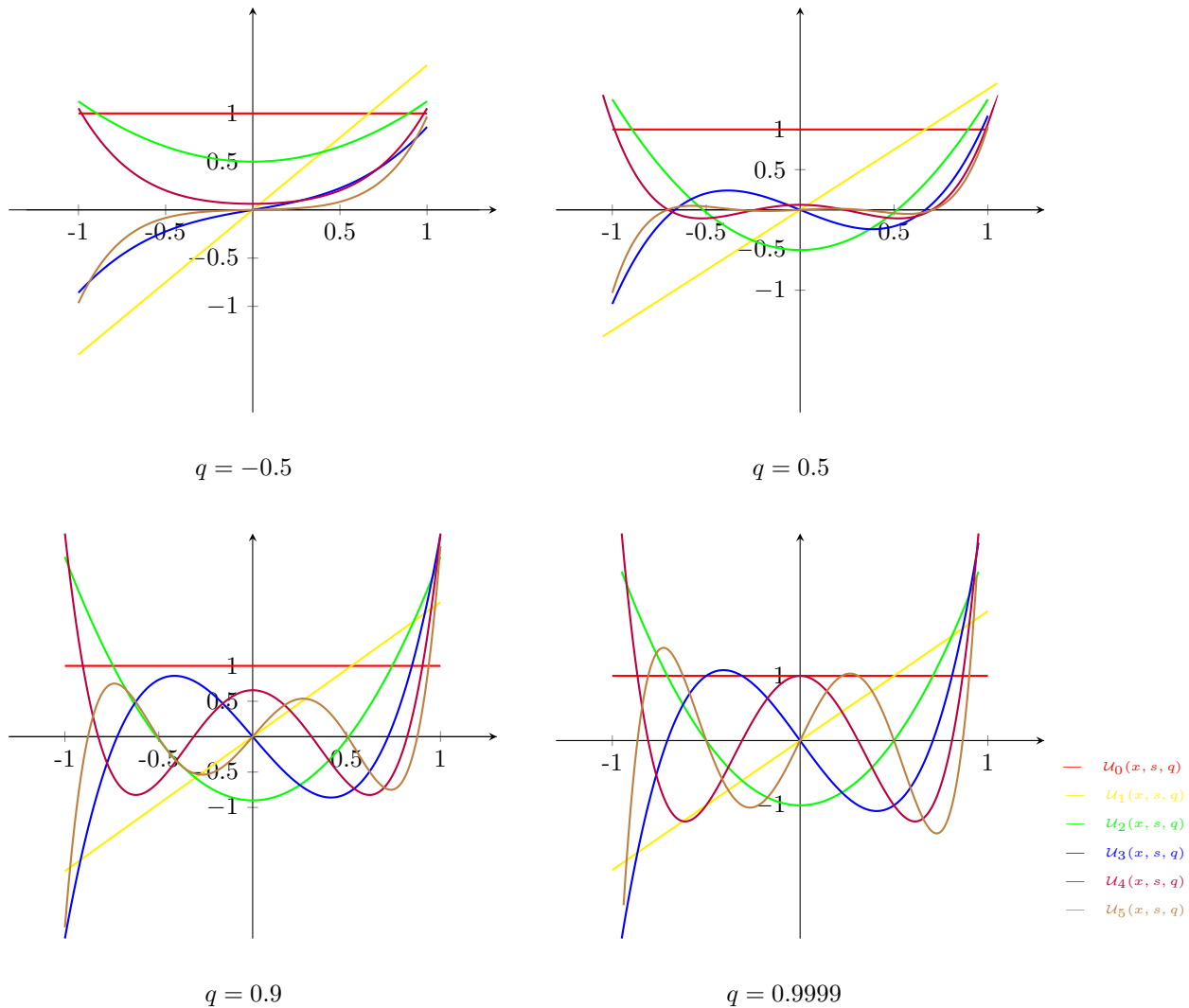
\begin{figure}[h!]
	\begin{tikzpicture}[scale=1.0]
	\begin{axis}[
	axis lines = center,
	ytick={-1,-0.5,0.5,1},
	xtick={-1,-0.5,0,0.5,1},
	xticklabels={-1,-0.5,0,0.5,1},
	]
	\addplot [
	domain=-1.299:1.3, 
	samples=300, 
	color=black,
	]
	{0};
	\addplot [
	domain=-1:1, 
	range=-1.6..1.6,
	samples=300, 
	color=red, thick,
	]
	{1};
	\addplot [
	domain=1:1.4,
	range=-1.6..1.6, 
	samples=300, 
	color=white, thin,
	]
	{1.5*x};
	\addplot [
	domain=-1.4:1,
	range=-1.6..1.6, 
	samples=300, 
	color=white, thin,
	]
	{1.5*x};
	\addplot [
	domain=-1:1,
	range=-1.6..1.6, 
	samples=300, 
	color=yellow, thick,
	]
	{1.5*x};
	\addplot [
	domain=-1:1, 
	samples=300, 
	color=green, thick,
	]
	{.625*x^2+.5};
	\addplot [
	domain=-1:1, 
	range=-1.6..1.6,
	samples=299, 
	color=blue, thick,
	]
	{.546875*x^3+.3125*x};
	\addplot [
	domain=-1:1,
	range=-1.6..1.6, 
	samples=299, 
	color=purple, thick,
	]
	{.5810546875*x^4+.41015625*x^2+.0625};     
	\addplot [
	domain=-1:1, 
	samples=299, 
	color=brown, thick,
	]
	{.5628967285*x^5+.3631591797*x^3+.041015625*x};  
	\end{axis}
	\draw[very thick] (3.3,-0.8) node {  $q=-0.5 $};
	\end{tikzpicture}
	\qquad
	\begin{tikzpicture}[scale=1.0]
	\begin{axis}[
	axis lines = center,
	ytick={-1,-0.5,0.5,1},
	xtick={-1,-0.5,0,0.5,1},
	xticklabels={-1,-0.5,0,0.5,1},
	]
	\addplot [
	domain=-1.299:1.3,
	range=-1.6..1.6, 
	samples=300, 
	color=black,
	]
	{0};
	\addplot [
	domain=-1:1, 
	range=-1.6..1.6,
	samples=300, 
	color=red, thick,
	]
	{1};
	\addplot [
	domain=-1.05:1.05,
	range=-1.6..1.6, 
	samples=300, 
	color=yellow, thick,
	]
	{1.5*x};
	\addplot [
	domain=-1:1,
	range=-1.6..1.6, 
	samples=300, 
	color=green, thick,
	]
	{1.875*x^2-.5};
	\addplot [
	domain=-1:1,
	range=-1.6..1.6, 
	samples=299, 
	color=blue, thick,
	]
	{2.109375*x^3-.9375*x};
	\addplot [
	domain=-1.05:1.05,
	range=-1.6..1.6, 
	samples=299, 
	color=purple, thick,
	]
	{2.241210938*x^4-1.23046875*x^2+.0625};     
	\addplot [
	domain=-1:1,
	range=-1.6..1.6, 
	samples=299, 
	color=brown, thick,
	]
	{2.311248780*x^5-1.400756836*x^3+.123046875*x};  
	
	\addplot [
	domain=-1.2:-1,
	range=-1.6..1.6, 
	samples=299, 
	color=white, thick,
	]
	{2.109375*x^3-.9375*x};
	\addplot [
	domain=1:1.2,
	range=-1.6..1.6, 
	samples=299, 
	color=white, thick,
	]
	{2.109375*x^3-.9375*x};
	\end{axis}
	\draw[very thick] (3.3,-0.8) node {  $q=0.5 $};
	\end{tikzpicture}
	\\
	\vskip 5pt
	\begin{tikzpicture}[scale=1.0]
	\begin{axis}[
	axis lines = center,
	ytick={-1,-0.5,0.5,1},
	xtick={-1,-0.5,0,0.5,1},
	xticklabels={-1,-0.5,0,0.5,1},
	]
	\addplot [
	domain=-1.299:1.3,
	range=-1.6..1.6, 
	samples=300, 
	color=black,
	]
	{0};
	\addplot [
	domain=-1:1, 
	range=-1.6..1.6,
	samples=300, 
	color=red, thick,
	]
	{1};
	\addplot [
	domain=-1:1,
	range=-1.6..1.6, 
	samples=300, 
	color=yellow, thick,
	]
	{1.9*x};
	\addplot [
	domain=-1:1,
	range=-1.6..1.6, 
	samples=300, 
	color=green, thick,
	]
	{3.439*x^2-.9};
	\addplot [
	domain=-1:1,
	range=-1.6..1.6, 
	samples=299, 
	color=blue, thick,
	]
	{5.946031*x^3-3.0951*x};
	\addplot [
	domain=-1:1,
	range=-1.6..1.6, 
	samples=299, 
	color=purple, thick,
	]
	{9.847221939*x^4-7.63282611*x^2+.6561};     
	\addplot [
	domain=-1:1,
	range=-1.6..1.6, 
	samples=299, 
	color=brown, thick,
	]
	{15.66190802*x^5-16.04112454*x^3+3.074215599*x};  
	\end{axis}
	\draw[very thick] (3.3,-0.8) node {   $q=0.9 $};
	\end{tikzpicture}
	\qquad
	\begin{tikzpicture}[scale=1.0]
	\begin{axis}[
	axis lines = center,
	ytick={-1,0,1},
	xtick={-1,0,1},
	xticklabels={-1,0,1},
	]
	\addplot [
	domain=-1.299:1.3,
	range=-1.6..1.6, 
	samples=300, 
	color=black,
	]
	{0};
	\addplot [
	domain=-1:1, 
	range=-1.6..1.6,
	samples=300, 
	color=red, thick,
	]
	{1};
	\addplot [
	domain=-1:1,
	range=-1.6..1.6, 
	samples=300, 
	color=yellow, thick,
	]
	{1.9999*x};
	\addplot [
	domain=-0.95:0.95,
	range=-1.6..1.6, 
	samples=300, 
	color=green, thick,
	]
	{3.999400040*x^2-.9999};
	\addplot [
	domain=-0.95:0.95,
	range=-1.6..1.6, 
	samples=299, 
	color=blue, thick,
	]
	{7.997600380*x^3-3.999000100*x};
	\addplot [
	domain=-0.95:0.95,
	range=-1.6..1.6, 
	samples=299, 
	color=purple, thick,
	]
	{15.99200220*x^4-11.99460118*x^2+.9996000600};     
	\addplot [
	domain=-0.94:0.94,
	range=-1.6..1.6, 
	samples=299, 
	color=brown, thick,
	]
	{31.97601000*x^5-31.97760808*x^3+5.996101160*x};  
	\end{axis}
	\draw[very thick] (8,2.4) node {\tiny {\color{red}{\textbf{---}}  \ \ $\:\mathcal{U}_{0}(x,s,q) \ $}};
	\draw[very thick] (8,2.0) node {\tiny {\color{yellow}{\textbf{---}}  \ \ $\:\mathcal{U}_{1}(x,s,q) $}};
	\draw[very thick] (8,1.6) node {\tiny {\color{green}{\textbf{---}}  \ \ $\:\mathcal{U}_{2}(x,s,q) $}};  
	\draw[very thick] (8,1.2) node {\tiny {\color{blue}{\textbf{---}} \ \ $\:\mathcal{U}_{3}(x,s,q) $}};
	\draw[very thick] (8,0.8) node {\tiny {\color{purple}{\textbf{---}} \ \ $\:\mathcal{U}_{4}(x,s,q) $}};
	\draw[very thick] (8,0.4) node {\tiny {\color{brown}{\textbf{---}} \ \ $\:\mathcal{U}_{5}(x,s,q) $}};
	\draw[very thick] (3.2,-0.8) node {   $q=0.9999 $};
	\end{tikzpicture}
	\vskip 5pt
	\caption{Graphs of  $\:\mathcal{U}_{n}(x,s,q)$\ for $s=-1$,\ $q=-0.5,0.5,0.9, 0.9999$, \ $n=0,1,2,3,4,5$ }
\end{figure}

In Figure 3 the graphs of the incomplete $q$-Chebyshev polynomials of second kind $\:\mathcal{U}_9^k(x,s,q)$ for $s=-1$,\ $q=-0.9, -0.5, 05, 0.9$,  $k=0,1,2,3,4$ are shown.

\begin{figure}[H]
	\begin{tikzpicture}[scale=1.0]
	\begin{axis}[
	axis lines = center,
	xtick={-2,-1.5,-1,-0.5,0,0.5,1,1.5,2},
	xticklabels={-2,-1.5,-1,-0.5,0,0.5,1,1.5,2},
	]
	\addplot [
	domain=-1.1:1.1, 
	samples=300, 
	color=black,
	]
	{0};
	\addplot [
	domain=-1:1,  
	samples=300, 
	color=black,
	thick,
	]
	{.0328965180*x^9+.1025671222*x^7+.1848556942*x^5+.1652186736*x^3+.06352168424*x};
	\addplot [
	domain=-1:1,  
	samples=300, 
	color=green,
	]
	{.02328965180*x^9+.1025671222*x^7+.1848556942*x^5+.1652186736*x^3};
	\addplot [
	domain=-1:1,  
	samples=300, 
	color=purple,
	]
	{.02328965180*x^9+.1025671222*x^7+.1848556942*x^5};    
	\addplot [
	domain=-1:1,   
	samples=300, 
	color=blue,
	]
	{.02328965180*x^9+.1025671222*x^7};    
	\addplot [
	domain=-1:1,  
	samples=300, 
	color=red,
	]
	{.02328965180*x^9};
	\end{axis}	
	\draw[very thick] (3.3,-0.8) node {  $q=-0.9 $};

	\end{tikzpicture}
	\qquad
	\begin{tikzpicture}[scale=1.0]
	\begin{axis}[
	axis lines = center,
	xtick={-2,-1.5,-1,-0.5,0,0.5,1,1.5,2},
	xticklabels={-2,-1.5,-1,-0.5,0,0.5,1,1.5,2},
	]
	\addplot [
	domain=-1:1.2, 
	samples=300, 
	color=black,
	]
	{0};
	\addplot [
	domain=-1:1, 
	samples=300, 
	color=black,
	thick,
	]
	{.8683291812*x^9+.3781446608*x^7+.05001999594*x^5+.001535305637*x^3+.00001016259192*x};
	\addplot [
	domain=-1:1, 
	samples=300, 
	color=green,
	]
	{.7683291812*x^9+.3781446608*x^7+.05001999594*x^5+.001535305637*x^3};
	\addplot [
	domain=-1:1, 
	samples=300, 
	color=purple,
	]
	{.6683291812*x^9+.3781446608*x^7+.05001999594*x^5};    
	\addplot [
	domain=-0.95:0.95, 
	samples=300, 
	color=blue,
	]
	{.5683291812*x^9+.3781446608*x^7};    
	\addplot [
	domain=-1:1, 
	samples=300, 
	color=red,
	]
	{.5683291812*x^9};
	\end{axis}	
	\draw[very thick] (3.3,-0.8) node {  $q=-0.5 $};
	
	\end{tikzpicture}
	\\
	\vskip 5pt
	\begin{tikzpicture}[scale=1.0]
	\begin{axis}[
	axis lines = center,
	xtick={-2,-1.5,-1,-0.5,0,0.5,1,1.5,2},
	xticklabels={-2,-1.5,-1,-0.5,0,0.5,1,1.5,2},
	]
	\addplot [
	domain=-1:1.2, 
	samples=300, 
	color=black,
	]
	{0};
	\addplot [
	domain=-1:1, 
	samples=300, 
	color=black,
	thick,
	]
	{2.679580385*x^9-1.577109807*x^7+.2053816423*x^5-.005921893170*x^3+.00003048777579*x};
	\addplot [
	domain=-1:1, 
	samples=300, 
	color=green,
	]
	{2.479580385*x^9-1.577109807*x^7+.2053816423*x^5-.005921893170*x^3};
	\addplot [
	domain=-1:1, 
	samples=300, 
	color=purple, 
	]
	{2.379580385*x^9-1.577109807*x^7+.2053816423*x^5};    
	\addplot [
	domain=-0.95:0.95, 
	samples=300, 
	color=blue, 
	]
	{2.379580385*x^9-1.577109807*x^7};    
	\addplot [
	domain=-1:1, 
	samples=300, 
	color=red,
	]
	{2.379580385*x^9};
	
	\end{axis}
	\draw[very thick] (3.3,-0.8) node {  $q=0.5 $};
	\end{tikzpicture}
	\qquad
	\begin{tikzpicture}[scale=1.0]
	\begin{axis}[
	axis lines = center,
	xtick={-2,-1.5,-1,-0.5,0,0.5,1,1.5,2},
	xticklabels={-2,-1.5,-1,-0.5,0,0.5,1,1.5,2},
	]
	\addplot [
	domain=-1:1.7, 
	samples=300, 
	color=black,
	]
	{0};
	\addplot [
	domain=-1.6:1.6,
	samples=300, 
	color=black,
	thick,
	]
	{76.37082070*x^9-136.8334111*x^7+87.03192062*x^5-20.02803928*x^3+1.206912000*x};
	\addplot [
	domain=-1.6:1.6,  
	samples=300, 
	color=green, 
	]
	{74.37082070*x^9-136.8334111*x^7+87.03192062*x^5-20.02803928*x^3};
	\addplot [
	domain=-1.6:1.6,  
	samples=300, 
	color=purple, 
	]
	{70.37082070*x^9-136.8334111*x^7+87.03192062*x^5};    
	\addplot [
	domain=-1.6:1.6, 
	samples=300, 
	color=blue, 
	]
	{70.37082070*x^9-136.8334111*x^7};    
	\addplot [
	domain=-1.6:1.6,  
	samples=300, 
	color=red, 
	]
	{70.37082070*x^9};
	
	\end{axis}
	\draw[very thick] (3.3,-0.8) node {  $q=0.9 $};
	\draw[very thick] (8,2.2) node {\tiny {\color{red}{\textbf{---}}  \ \ $\:\mathcal{U}_{9}^{0}(x,s,q) \ $}};
	\draw[very thick] (8,1.7) node {\tiny {\color{blue}{\textbf{---}}  \ \ $\:\mathcal{U}_{9}^{1}(x,s,q) $}};
	\draw[very thick] (8,1.2) node {\tiny {\color{purple}{\textbf{---}}  \ \ $\:\mathcal{U}_{9}^{2}(x,s,q) $}};  
	\draw[very thick] (8,0.7) node {\tiny {\color{green}{\textbf{---}} \ \ $\:\mathcal{U}_{9}^{3}(x,s,q) $}}; \ \
	\draw[very thick] (8,0.2) node {\tiny {\color{black}{\textbf{---}} \ \ $\:\mathcal{U}_{9}^{4}(x,s,q) $}};
	\end{tikzpicture}
	\vskip 5pt
	\caption{Graphs of  $\:\mathcal{U}_{9}^{k}(x,s,q)$ for \ $s=-1$, \ $q=-0.9,-0.5,0.5,0.9$,\  $k=0,1,2,3,4$ }
\end{figure}
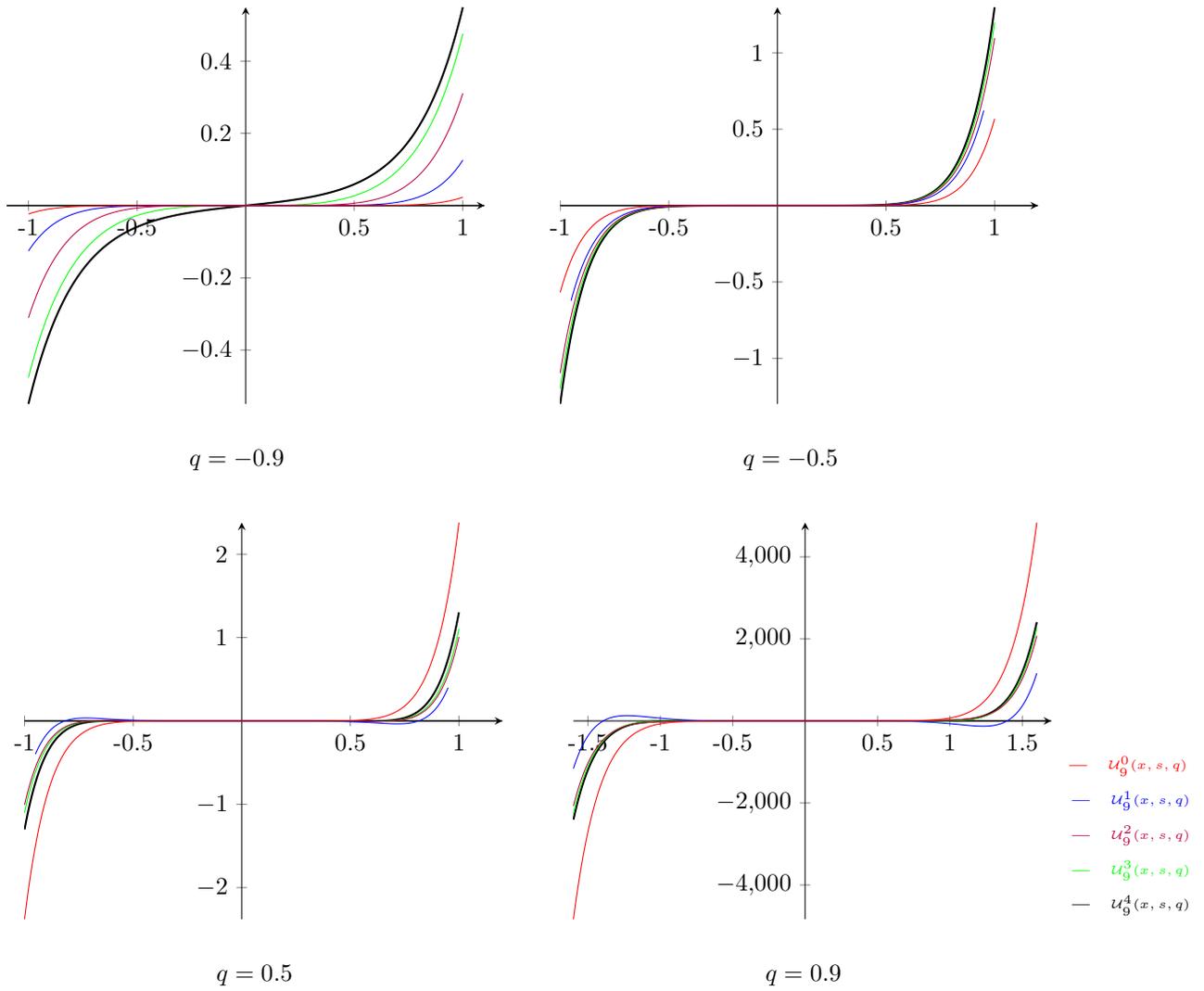

In Figure 4 the graphs of the incomplete Lucas polynomials $\mathcal{T}_5^k(\tfrac{x}{2},s,q)$ for $s=1$,\ $q=-0.9, -0.5, 05, 0.9$,  $k=0,1,2$ are shown.

\begin{figure}[h!]

	\begin{tikzpicture}[scale=0.9]
	\begin{axis}[
	axis lines = center,
	]
	\addplot [
	domain=-6:6.2,
	samples=300, 
	color=black,
	]
	{0};
	\addplot [
	domain=-6:6,  
	samples=300, 
	color=black,
	thick,
	]
	{.002538542534*x^5-.04619316612*x^3+.2746106550*x};
	\addplot [
	domain=-6:6,
	samples=300, 
	color=blue,
	]
	{.002538542534*x^5-.04619316612*x^3};    
	\addplot [
	domain=-6:6, 
	samples=300, 
	color=red,
	]
	{.002538542534*x^5};
	
	\end{axis}
	\draw[very thick] (3.3,-0.8) node {  $q=-0.9 $};
	\end{tikzpicture}
	\qquad 
	\begin{tikzpicture}[scale=0.9]
	\begin{axis}[
	axis lines = center,
	]
	\addplot [
	domain=-2:2.2,
	samples=300, 
	color=black,
	]
	{0};
	\addplot [
	domain=-2:2, 
	samples=300, 
	color=black,
	thick,
	]
	{.01815795898*x^5-.04699707030*x^3+.02148437500*x};
	\addplot [
	domain=-2:2,
	samples=300, 
	color=blue,
	]
	{.01815795898*x^5-.04699707030*x^3};    
	\addplot [
	domain=-2:2, 
	samples=300, 
	color=red,
	]
	{.01815795898*x^5};
	
	\end{axis}
	\draw[very thick] (3.3,-0.8) node {  $q=-0.5 $};
	\end{tikzpicture}
	\\
	\vskip 5pt
	\begin{tikzpicture}[scale=0.9]
	\begin{axis}[
	axis lines = center,
	]
	\addplot [
	domain=-2:2.2,
	samples=300, 
	color=black,
	]
	{0};
	\addplot [
	domain=-2:2, 
	samples=300, 
	color=black,
	thick,
	]
	{.07003784181*x^5+.1702880860*x^3+.06054687500*x};
	\addplot [
	domain=-2:2,
	samples=300, 
	color=blue,
	]
	{.07003784181*x^5+.1702880860*x^3};    
	\addplot [
	domain=-2:2, 
	samples=300, 
	color=red,
	]
	{.07003784181*x^5};
	
	\end{axis}
	\draw[very thick] (3.3,-0.8) node {  $q=0.5 $};
	\end{tikzpicture}
	\qquad
	\begin{tikzpicture}[scale=0.9]
	\begin{axis}[
	axis lines = center,
	]
	\addplot [
	domain=-2:2.2,
	samples=300, 
	color=black,
	]
	{0};
	\addplot [
	domain=-2:2, 
	samples=300, 
	color=black,
	thick,
	]
	{.3077256856*x^5+1.441752131*x^3+1.343397555*x};
	\addplot [
	domain=-2:2,
	samples=300, 
	color=blue,
	]
	{.3077256856*x^5+1.441752131*x^3};    
	\addplot [
	domain=-2:2, 
	samples=300, 
	color=red,
	]
	{.3077256856*x^5};
	
	\end{axis}
	\draw[very thick] (8,1.8) node {\tiny {\color{red}{\textbf{---}}  \ \ $\mathcal{T}_{5}^{0}(x,s,q) \ $}};
	\draw[very thick] (8,1.2) node {\tiny {\color{blue}{\textbf{---}}  \ \ $\mathcal{T}_{5}^{1}(x,s,q) $}};
	\draw[very thick] (8,0.6) node {\tiny {\color{black}{\textbf{---}} \ \ $\mathcal{T}_{5}^{2}(x,s,q) $}};
	\draw[very thick] (3.3,-0.8) node {  $q=0.9 $};
	\end{tikzpicture}
	\vskip 5pt
	\caption{Graphs of  $\mathcal{T}_{5}^{k}(\frac{x}{2},s,q)$ for \ $s=1$,\ $q=-0.9,-0.5,0.5,0.9$, \  $k=0,1,2$  }
\end{figure}
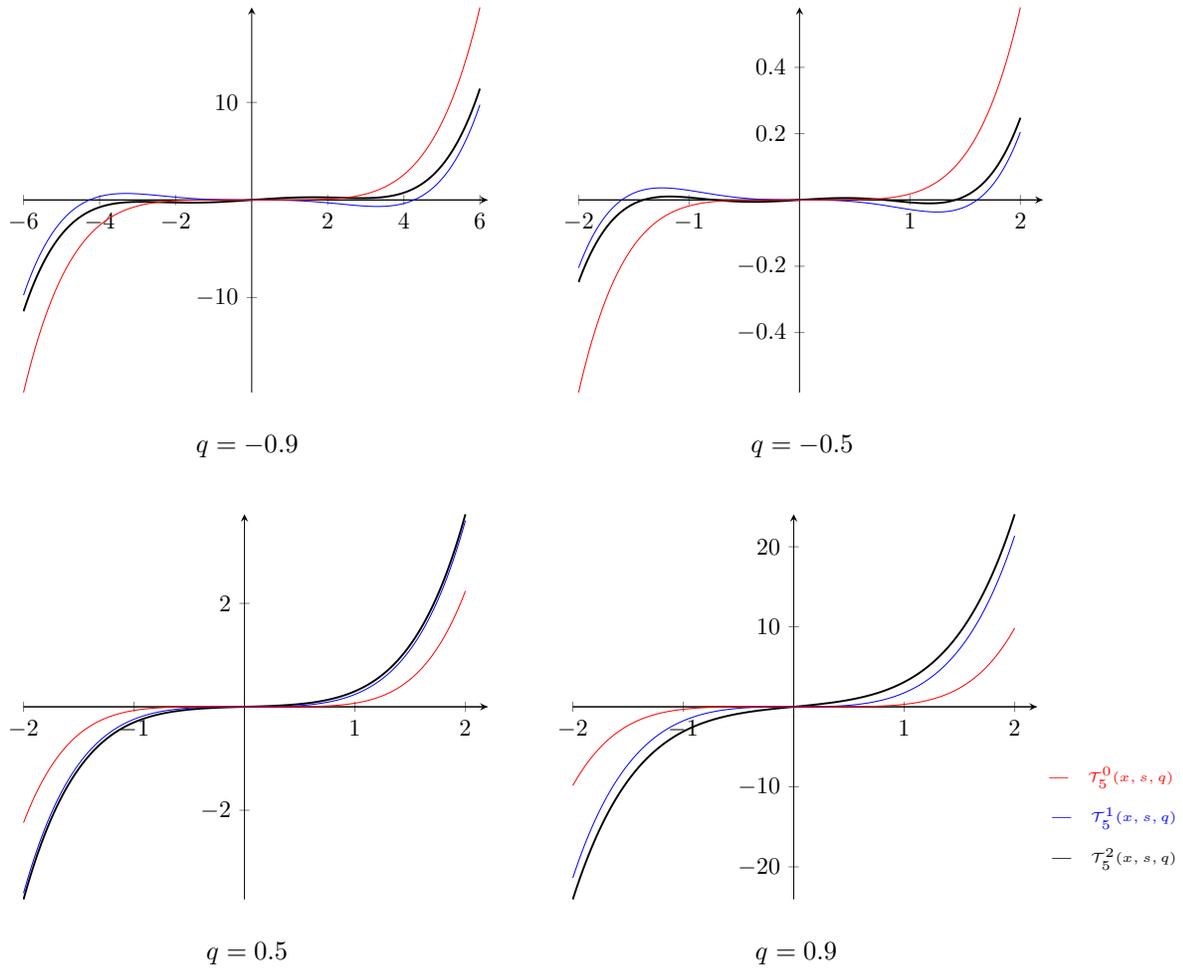

In Figure 5 the graphs of the incomplete Jacobsthal polynomials $\:\mathcal{U}_8^k(\tfrac{x}{2},s,q)$ for $s=2$,\ $q=-0.9, -0.5, 05, 0.9$,  $k=0,1,2,3,4$ are shown.

\begin{figure}[h!]

	\begin{tikzpicture}[scale=0.8]
	\begin{axis}[
	axis lines = center,
	]
	\addplot [
	domain=-8:8.2, 
	samples=300, 
	color=black,
	]
	{0};
	\addplot [
	domain=-8:8,  
	samples=300, 
	color=black,
	thick,
	]
	{.0001038203901*x^8-.004294937688*x^6+.07564372025*x^4-.6962141485*x^2+2.964832302};
	\addplot [
	domain=-8:8,  
	samples=300, 
	color=green,
	]
	{.0001038203901*x^8-.004294937688*x^6+.07564372025*x^4-.6962141485*x^2};
	\addplot [
	domain=-8:8,  
	samples=300, 
	color=purple,
	]
	{.0001038203901*x^8-.004294937688*x^6+.07564372025*x^4};    
	\addplot [
	domain=-8:8, 
	samples=300, 
	color=blue,
	]
	{.0001038203901*x^8-.004294937688*x^6};    
	\addplot [
	domain=-8:8, 
	samples=300, 
	color=red,
	]
	{.0001038203901*x^8};
	
	\end{axis}
	\draw[very thick] (3.3,-0.8) node {  $q=-0.9 $};
	\end{tikzpicture}
	\qquad
	\begin{tikzpicture}[scale=0.8]
	\begin{axis}[
	axis lines = center,
	]
	\addplot [
	domain=-2.5:2.7, 
	samples=300, 
	color=black,
	]
	{0};
	\addplot [
	domain=-2.5:2.5, 
	samples=300, 
	color=black,
	thick,
	]
	{.0209277116512*x^8+.04890039106*x^6+.04974390265*x^4+.01036584377*x^2+.0002441406250};
	\addplot [
	domain=-2.5:2.5, 
	samples=300, 
	color=green,
	]
	{.009277116512*x^8+.04890039106*x^6+.04974390265*x^4+.01036584377*x^2};
	\addplot [
	domain=-1:1, 
	samples=300, 
	color=purple,
	]
	{.009277116512*x^8+.04890039106*x^6+.04974390265*x^4};    
	\addplot [
	domain=-2.5:2.5,
	samples=300, 
	color=blue,
	]
	{.009277116512*x^8+.04890039106*x^6};    
	\addplot [
	domain=-2.5:2.5, 
	samples=300, 
	color=red,
	]
	{.009277116512*x^8};
	
	\end{axis}
	\draw[very thick] (3.3,-0.8) node {  $q=-0.5 $};
	\end{tikzpicture}
	\\
	\vskip 5pt
	\begin{tikzpicture}[scale=0.8]
	\begin{axis}[
	axis lines = center,
	]
	\addplot [
	domain=-2.5:2.7, 
	samples=300, 
	color=black,
	]
	{0};
	\addplot [
	domain=-2.5:2.5, 
	samples=300, 
	color=black,
	thick,
	]
	{.1981274030*x^8+2.738223873*x^6+11.55308595*x^4+15.12770918*x^2+2.964832302};
	\addplot [
	domain=-1:1, 
	samples=300, 
	color=green,
	]
	{.1981274030*x^8+2.738223873*x^6+11.55308595*x^4+15.12770918*x^2};
	\addplot [
	domain=-2.5:2.5, 
	samples=300, 
	color=purple,
	]
	{.1981274030*x^8+2.738223873*x^6+11.55308595*x^4};    
	\addplot [
	domain=-2.5:2.5, 
	samples=300, 
	color=blue,
	]
	{.1981274030*x^8+2.738223873*x^6};    
	\addplot [
	domain=-2.5:2.5, 
	samples=300, 
	color=red,
	]
	{.1981274030*x^8};
	
	\end{axis}
	\draw[very thick] (3.3,-0.8) node {  $q=0.5 $};
	\end{tikzpicture}
	\qquad
	\begin{tikzpicture}[scale=0.8]
	\begin{axis}[
	axis lines = center,
	]
	\addplot [
	domain=-2.5:2.7, 
	samples=300, 
	color=black,
	]
	{0};
	\addplot [
	domain=-2.6:2.6, 
	samples=300, 
	color=black,
	thick,
	]
	{.0023380357*x^8-.01190952285*x^6+.01289656735*x^4-.003455281255*x^2+.0002441406250};
	\addplot [
	domain=-2.5:2.5, 
	samples=300, 
	color=green,
	]
	{0.002222380357*x^8-0.01190952285*x^6+0.01289656735*x^4-0.003455281255*x^2};
	\addplot [
	domain=-2.5:2.5,  
	samples=300, 
	color=purple,
	]
	{0.002224380357*x^8-0.01190952285*x^6+0.01289656735*x^4};    
	\addplot [
	domain=-2.5:2.5, 
	samples=300, 
	color=blue,
	]
	{0.002224380357*x^8-0.01190952285*x^6};    
	\addplot [
	domain=-2.5:2.5,  
	samples=300, 
	color=red,
	]
	{0.002224380357*x^8};
	
	\end{axis}
	\draw[very thick] (3.3,-0.8) node {  $q=0.9 $};
	\draw[very thick] (8,2.9) node {\tiny {\color{red}{\textbf{---}}  \ \ $\:\mathcal{U}_{8}^{0}(x,s,q) \ $}};
	\draw[very thick] (8,2.3) node {\tiny {\color{blue}{\textbf{---}}  \ \ $\:\mathcal{U}_{8}^{1}(x,s,q) $}};
	\draw[very thick] (8,1.7) node {\tiny {\color{purple}{\textbf{---}}  \ \ $\:\mathcal{U}_{8}^{2}(x,s,q) $}};  
	\draw[very thick] (8,1.1) node {\tiny {\color{green}{\textbf{---}} \ \ $\:\mathcal{U}_{8}^{3}(x,s,q) $}}; \ \
	\draw[very thick] (8,0.4) node {\tiny {\color{black}{\textbf{---}} \ \ $\:\mathcal{U}_{8}^{4}(x,s,q) $}};
	\end{tikzpicture}
	\vskip 5pt
	\caption{Graphs of $\:\mathcal{U}_{8}^{k}(\frac{x}{2},s,q)$ \ $s=2$, $q=-0.9,-0.5,0.5,0.9$,  $k=0,1,2,3,4$ }
\end{figure}
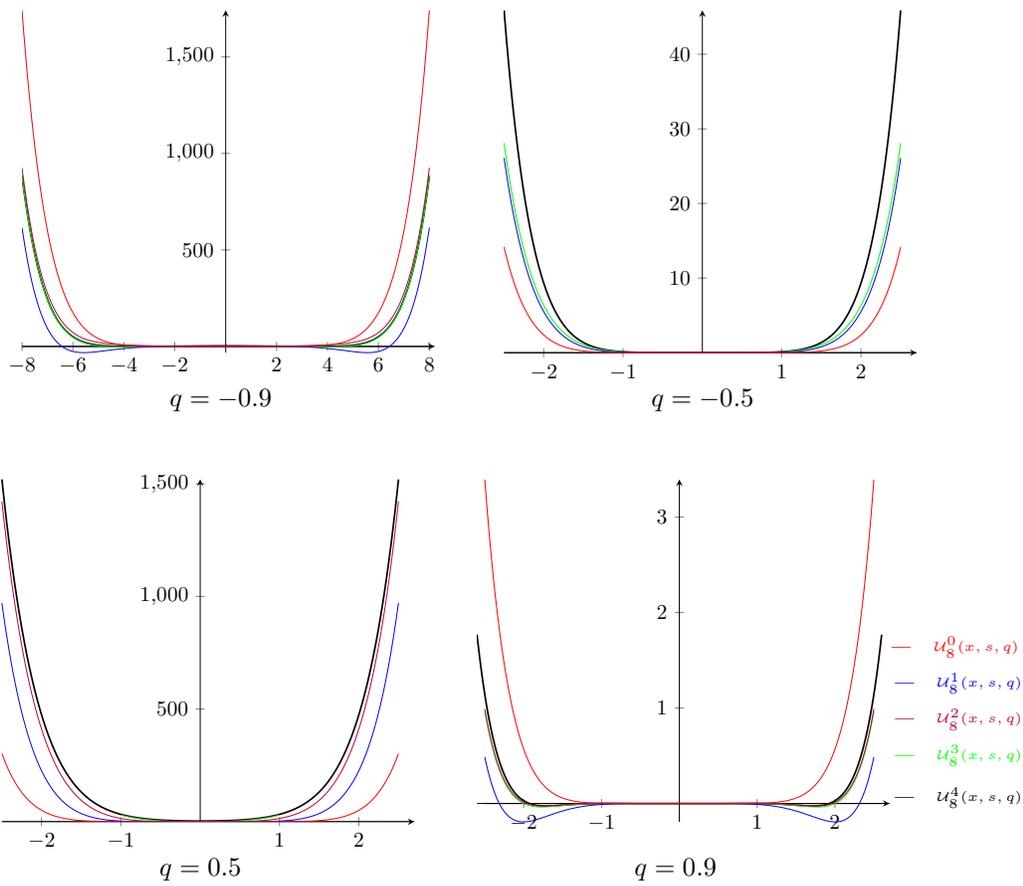

\vskip 5pt

In Figure 6, the graphs of incomplete Fibonacci numbers $\:\mathcal{U}_{n}^{k}(\frac{x}{2},1,1)$ and incomplete Lucas numbers $\mathcal{T}_{n}^{k}(\frac{x}{2},1,1)$ for $1\leq n \leq 9$  $0 \leq k \leq k$ are shown.
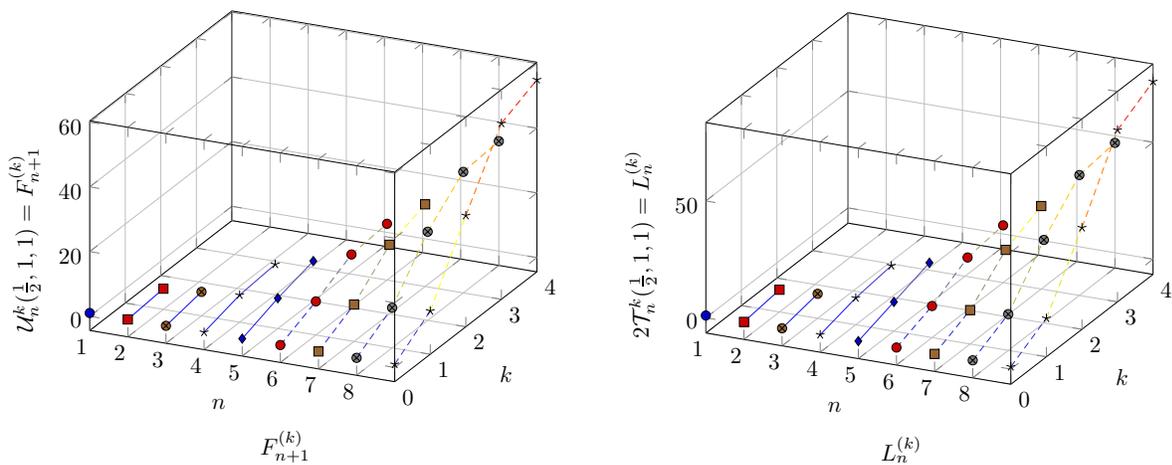
\begin{figure}[h!]
	\centering
	\resizebox {\columnwidth} {!} {
		\begin{tikzpicture}[scale=1]
		\begin{axis}[
		3d box=complete,
		grid=major,
		samples=5, domain=1:5,
		xtick={0,1,...,8},ytick={0,1,...,4},
		xlabel={$n$},
		ylabel={$k$},
		zlabel={$\:\mathcal{U}_n^k(\frac{1}{2},1,1)=F_{n+1}^{(k)}$},
		]
		\addplot3+[mesh] coordinates {(1,0,1) };
		\addplot3+[mesh] coordinates { (2,0,1) (2,1,2)  };
		\addplot3+[mesh] coordinates { (3,0,1)  (3,1,3) };
		\addplot3+[mesh] coordinates { (4,0,1)  (4,1,4) (4,2,5)};
		\addplot3+[mesh] coordinates { (5,0,1)  (5,1,5) (5,2,8)};
		\addplot3+[mesh] coordinates { (6,0,1)  (6,1,6) (6,2,12) (6,3,13) };
		\addplot3+[mesh] coordinates {  (7,0,1)  (7,1,7) (7,2,17) (7,3,21) };
		\addplot3+[mesh] coordinates {  (8,0,1)  (8,1,8) (8,2,23) (8,3,33) (8,4,34) };
		\addplot3+[mesh] coordinates {  (9,0,1)  (9,1,9) (9,2,30) (9,3,50) (9,4,55) };
		\end{axis}
		\draw[very thick] (3,-1) node {   $F_{n+1}^{(k)}$};
		\end{tikzpicture}
		\qquad
		\begin{tikzpicture}[scale=1]
		\begin{axis}[
		3d box=complete,
		grid=major,
		samples=5, domain=0:7,
		xtick={0,1,...,8},ytick={0,1,...,4},
		xlabel={$n$},
		ylabel={$k$},
		zlabel={$2\mathcal{T}_n^k(\frac{1}{2},1,1)=L_n^{(k)}$},
		]
		\addplot3+[mesh] coordinates {(1,0,1) };
		\addplot3+[mesh] coordinates { (2,0,1) (2,1,3)  };
		\addplot3+[mesh] coordinates { (3,0,1)  (3,1,4) };
		\addplot3+[mesh] coordinates { (4,0,1)  (4,1,5) (4,2,7)};
		\addplot3+[mesh] coordinates { (5,0,1)  (5,1,6) (5,2,11)};
		\addplot3+[mesh] coordinates { (6,0,1)  (6,1,7) (6,2,16) (6,3,18) };
		\addplot3+[mesh] coordinates {  (7,0,1)  (7,1,8) (7,2,22) (7,3,29) };
		\addplot3+[mesh] coordinates {  (8,0,1)  (8,1,9) (8,2,29) (8,3,45) (8,4,47) };
		\addplot3+[mesh] coordinates {  (9,0,1)  (9,1,10) (9,2,37) (9,3,67) (9,4,76) };
		\end{axis}
		\draw[very thick] (3,-1) node {  $L_n^{(k)}$};
		\end{tikzpicture}
		\hspace{1.5cm}
	}
	\vskip 5pt
	\caption{Graphs of incomplete Fibonacci ve Lucas numbers for \ $1\leq n\leq 9$, \  $0 \leq k \leq4$. }
\end{figure}

\clearpage



\begin{thebibliography}{99}
	
	\bibitem{1} Aral, A. , Gupta, V., Agarwal, R., Applications of $q$-calculus in operator theory, Springer, (2013).
	
	\bibitem{2} Carlitz, L., Fibonacci notes 4: $q$-Fibonacci polynomials,
	Fibonacci Quart., 13, 97-102, (1975).
	
	\bibitem{3} Cigler, J., A simple approach to $q$-Chebyshev polynomials,  arXiv:1201.4703v2, (2012).
	
	\bibitem{4} Cigler, J., $q-$Chebyshev polynomials, arXiv: 1205.5383, (2012).
	
	\bibitem{5} Cigler, J., A new class of $q$-Fibonacci polynomials, The Electron. J. Combin. 10 (2003) R19.
	
	\bibitem{6} Pan, H., Arithmetic properties of $q$-Fibonacci numbers and $q$-pell numbers, Discrete Math., 306, 2118–2127, (2006).
	
	\bibitem{9} Djordjevi\'{c}, G.B. and Srivastava, H.M., Incomplete generalized Jacobsthal and Jacobsthal Lucas numbers, Math. Comput. Modelling, 42(9-10), 1049-1056, (2005).
	
	\bibitem{10} Djordjevi\'{c}, G.B., Generating functions of the incomplete generalized Fibonacci and generalized Lucas numbers, Fibonacci Quart., 42(2), 106-113, (2004).
	
	\bibitem{11} Filipponi, P., Incomplete Fibonacci and Lucas numbers, Rend. Circ. Mat. Palermo, 45(2), 37-56, (1996).
	
	\bibitem{12} Mason, J.C. and Handscomb, J.C., Chebyshev Polynomials, Chapman \& Hall, (2003).
	
	\bibitem{13} Pint\'{e}r, A. and Srivastava, H.M., Generating functions of the incomplete Fibonacci and Lucas numbers, Rend. Circ. Mat. Palermo, 48(2), 591-596, (1999).
	
	\bibitem{14} Ramirez, J.L. and Sirvent V.F., A $q$-analogue of the Biperiodic Fibonacci Sequence, arXiv:1501.05830, (2015).
	
	\bibitem{15} Tasci, D., Cetin Firengiz, M. and Tuglu, N., Incomplete bivariate Fibonaccu and Lucas $p$-polynomials, Discrete Dyn. Nat. Soc., vol. 2012, Article ID 840345, 11 pages, (2012).
	
	\bibitem{16} Tasci, D. and Cetin Firengiz, M., Incomplete Fibonacci and Lucas $p$-numbers, Math. Comput. Modelling, 52(9-10), 1763-1770, (2010).
	
	\bibitem{17} Andrews, G.E., Fibonacci numbers and the
	Rogers-Ramanujan identities. Fibonacci Quart., 42, 3-19, (2004).
	
	\bibitem{18} Kim, T., Some identities for the Bernoulli, the Euler and the
	Genocchi numbers and polynomials, Adv. Stud. Contemp. Math. (Kyungshang), vol. 20, no. 1, pp. 23-28, (2010).
	
	\bibitem{19} Govil, N.K. and Gupta, V., Convergence of $q$-Meyer-K\"{o}nig-Zeller-Durrmeyer operators, Adv. Stud. Contemp. Math. (Kyungshang), vol. 19, no. 1, pp. 97-108, (2009).
	
	\bibitem{20} Phillips, G.M., Interpolation and approximation by polynomials, Springer-Verlag, (2003).
	
	
\end{thebibliography}
\end{document}